\documentclass{amsart}

\newtheorem{theo}{Theorem}[section]

\newtheorem{lemma}[theo]{Lemma}

\newtheorem{theorem}[theo]{Theorem}

\theoremstyle{definition}

\theoremstyle{remark}
\newtheorem{remark}[theo]{Remark}
\newtheorem{example}[theo]{Example}

\usepackage{tikz}
\usetikzlibrary{cd}
\usepackage{here}
\usepackage{amssymb}
\usepackage{enumitem}
\usepackage{scalerel}
\usepackage{url}
\usepackage{cite}

\title{Links and singularities of stable maps from 3-manifolds to the plane}

\author{Gakuto Kato}
\address{Graduate School of Integrated Basic Sciences, Nihon University, 3-25-40 Sakurajosui, Setagaya-ku, Tokyo 156-8550, JAPAN}
\email{chga24002@g.nihon-u.ac.jp}

\date{\today}

\begin{document}

\keywords{braid, link, stable map}

\subjclass[2020]{57R45, 57M99, 57K10}

\begin{abstract}

For any two links in the $3$-sphere, we provide a visual construction of a stable map $f$ from the $3$-sphere to the Euclidean plane such that $f$ has no cusp points, the set of definite (resp. indefinite) fold points of $f$ is isotopic to the first (resp. second) given link, and $f$ does not have a certain type of singular fibers.


\end{abstract}

\maketitle


\section{Introduction}
\label{s1}

In this paper, all manifolds and maps are assumed to be differentiable of class $C^\infty$ unless otherwise indicated.

There have been many studies focusing on the singularities of stable maps from a closed orientable $3$-manifold $M$ to the Euclidean plane $\mathbb{R}^{2}$. For each map, the set of its \textit{singular points} consists of points where the differential is not surjective. See \cite{Furutani-Koda, Hiratuka-Jorge-Saeki, Ishikawa-Koda, Kalmar-Stipsicz, K.-Levine-Port., Levin'65, Saeki1993, Saeki1994, Saeki'95, Saeki'96, Saeki2019, Saeki-Yamamoto2016, Saeki-Yamamoto2018, S.Naoki, Ichihara-K}, for example. In particular, it was shown by Saeki \cite[Corollary 6.3]{Saeki'95, Saeki'96} that, for a closed orientable $3$-manifold $M$ and any given link $L$ in $M$, there exists a stable map $f : M \to \mathbb{R}^{2}$ such that the set of singular points, say $S(f)$, is equal to $L$ if and only if $[L]_{2}=0$ in $H_{1} (M;\mathbb{Z}_{2})$. For instance, it implies that every nonempty knot or link in the $3$-sphere $S^3$ is realized as the singular set $S(f)$ of a stable map $f : S^3 \to \mathbb{R}^{2}$. For such a stable map $f$, it is known that the points in $S(f)$ are classified into definite fold points, indefinite fold points, and cusp points. Further, it is known that singular fibers of $f$ containing two indefinite fold points are classified into \textit{type $\mathrm{I\hspace{-1.2pt}I^{2}}$} and \textit{type $\mathrm{I\hspace{-1.2pt}I^{3}}$} (Figure~\ref{types of sing. fibers}). See \cite{Levine1985, Levin'65}, for example. We remark that Saeki and Yamamoto~\cite{Saeki-Yamamoto2016} define a singular fiber as a map germ, but it represents the same object as ours.

    \begin{figure}[htbp]
        {\unitlength=1cm
        \begin{picture}(12.5,2)(0,0)
        \put(4,-0.5){\includegraphics[height=3cm,clip]{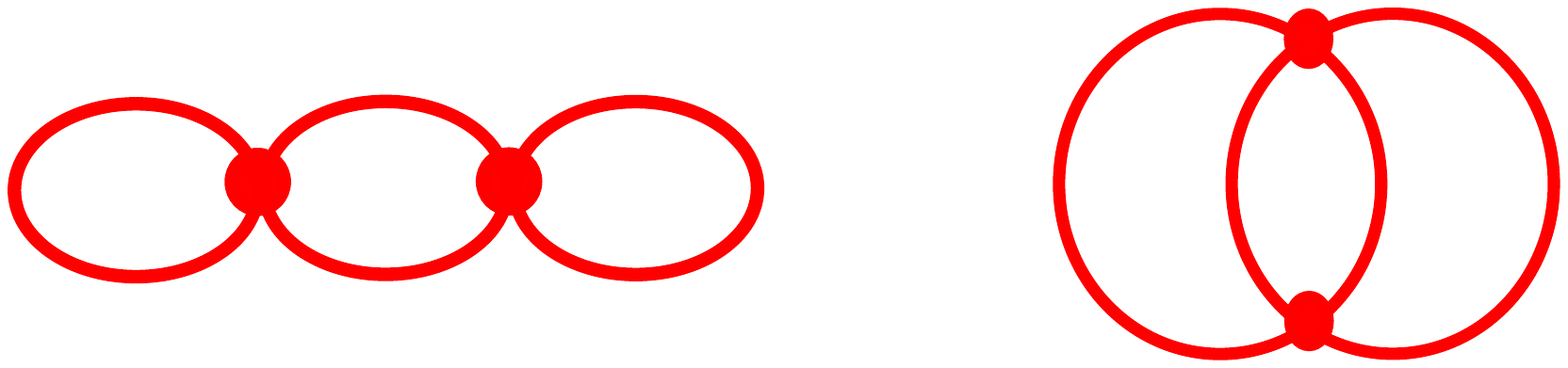}}
        \put(4.5,0.2){type $\mathrm{I\hspace{-1.2pt}I^{2}}$}
        \put(7,0.2){type $\mathrm{I\hspace{-1.2pt}I^{3}}$}
        \end{picture}}
        \caption{Types of singular fibers containing two indefinite fold points.}
        \label{types of sing. fibers}
    \end{figure}

 In the following, let $S_{0} (f)$ be the set of definite fold points of a stable map $f$ and $S_{1} (f)$ the set of indefinite fold points of $f$ from a $3$-manifold $M$ to a $2$-manifold. Note that it is known that $S_{0} (f)$ and $S_{1} (f)$ are $1$-dimensional submanifolds of $M$. In \cite{K2026}, the author obtains a result focusing on the set of definite fold points and singular fibers. Let $L$ be any link in $S^{3}$. Then, it is shown that there exists a stable map $f$ from $S^{3}$ to $\mathbb{R}^{2}$ such that $f$ has no cusp points, $S_{0}(f)$ is isotopic to $L$, and $f$ has no singular fibers of type $\mathrm{I\hspace{-1.2pt}I^{3}}$. As a continuation of the above work, in this paper, we give a result focusing on the set of indefinite fold points $S_{1} (f)$ and singular fibers.

\begin{theorem}
        \label{main1}
            Let $L$ be a link in $S^{3}$. Then, there exists a stable map $f : S^{3} \to \mathbb{R}^{2}$ such that $f$ has no cusp points and no singular fibers of type $\mathrm{I\hspace{-1.2pt}I^{3}}$, and $S_{1}(f)$ is isotopic to $L$. Moreover, if $L$ is represented as the closure of an $n$-braid $\sigma_{z_{1}}^{m_{1}} \sigma_{z_{2}}^{m_{2}} \cdots \sigma_{z_{l}}^{m_{l}}$ with $n \ge 2$,  $1 \leq z_{i} \leq n-1$, $z_{i-1} \neq z_{i}$ and $m_{i}$ non-zero integers, then $|\mathrm{I\hspace{-1.2pt}I^{2}}(f)| \le 4l- 3 | X |$ holds with $X = \{ i \mid z_{i}=1,~ 1\le i \le l \}$.
            \label{1}
\end{theorem}



    

Moreover, using \cite[Theorem 1.1]{K2026}, we obtain the following.

\begin{theorem}
    \label{main2}
        Let $L_{0}$ and $L_{1}$ be links in $S^{3}$. Then, there exists a stable map $f : S^{3} \to \mathbb{R}^{2}$ such that $f$ has no cusp points and no singular fibers of type $\mathrm{I\hspace{-1.2pt}I^{3}}$, $S_{0} (f)$ is isotopic to $L_{0}$, and $S_{1} (f)$ is isotopic to $L_{1}$.
\end{theorem}


\begin{remark}

The following are the results related to the results obtained in this paper. As a part of \cite{Ishikawa-Koda}, if a link $L$ in $S^{3}$ is obtained by Dehn filling on a specific model link, then there exists a stable map $f : S^{3} \to \mathbb{R}^{2}$ such that $f$ has no cusp points and no singular fibers of type $\mathrm{I\hspace{-1.2pt}I^{3}}$, $S_{0} (f)$ contains $L$ and $f$ has only a unique singular fiber of type $\mathrm{I\hspace{-1.2pt}I^{2}}$. As a part of \cite{Furutani-Koda}, if a link $L$ in $S^{3}$ is obtained by Dehn filling on a specific model link, then there exists a stable map $f : S^{3} \to \mathbb{R}^{2}$ such that $f$ has no cusp points and no singular fibers of type $\mathrm{I\hspace{-1.2pt}I^{2}}$, $S_{0} (f)$ contains $L$ and $f$ has only a unique singular fiber of type $\mathrm{I\hspace{-1.2pt}I^{3}}$. As a part of \cite{Ichihara-K}, if a two-bridge link $L$ in $S^{3}$ is given by a Conway form $C(a_{1}, b_{1}, \dots, a_{m}, b_{m}, a_{m+1})$ with non-zero integers $a_{i}, b_{i}$ where each $b_{i}$ is even, then there exists a stable map $f : S^{3} \to \mathbb{R}^{2}$ such that $S_{0}(f)$ is isotopic to $L$, $f$ has $2m$ singular fibers of type $\mathrm{I\hspace{-1.2pt}I^{2}}$, and no cusp points and no singular fibers of type $\mathrm{I\hspace{-1.2pt}I^{3}}$. As a part of \cite{Kalmar-Stipsicz}, if a $3$-manifold $M$ is obtained by integral surgery on a given link $L$ in $S^{3}$, then there is a stable map $f : M \to \mathbb{R}^{2}$ such that $S_{1} (f)$ contains a link in $S^3 \setminus N_L$, which is isotopic to $L$ in $S^3$.

\end{remark}


\section{The proof of Theorem~\ref{main1}}
\label{s2}

In this section, we give a proof of Theorem~\ref{main1}. Before proceeding to the proof, we introduce several definitions and known facts, and establish Lemma~\ref{sub1} for later use.

For smooth manifolds $M$ and $N$, let $C^{\infty}(M,N)$ be the set of smooth maps from $M$ to $N$ with the Whitney topology. A smooth map $f : M \to N$ is called a \textit{stable map} if there exists a neighborhood $U_{f}$ of $f$ in $C^{\infty}(M,N)$ such that, for any map $g$ in $U_{f}$, there are diffeomorphisms $\Phi : M \to M$ and $\phi: N \to N$ satisfying $g = \phi \circ f \circ \Phi^{-1}$. In the case where the source and target dimensions are $3$ and $2$, a characterization of stable maps is given as follows (cf. \cite{Levin'65}). For a closed orientable $3$-manifold $M$, a smooth map $f : M \to \mathbb{R}^{2}$ is a stable map if and only if $f$ is locally described as in one of the following forms:
\begin{enumerate}
    \item $(u,x,y) \mapsto (u,x)$,
    \item $(u,x,y) \mapsto (u,x^{2} + y^{2})$,
    \item $(u,x,y) \mapsto (u,x^{2} - y^{2})$,
    \item $(u,x,y) \mapsto (u,y^{2} + ux - x^{3})$,
\end{enumerate}
and $f$ globally satisfies

\begin{enumerate}[resume]
    \item $f^{-1}(f(p)) \cap S(f) = \{p\}$ for a point $p$ around which $f$ is locally expressed as (4),
    \item except for the points locally expressed as (4), the restriction of $f$ to $S(f)$ is an immersion with only normal crossings.
\end{enumerate}

The points around which $f$ is described as (2), (3), and (4) are the singular points of $f$, and they are called \textit{definite fold} points, \textit{indefinite fold} points, and \textit{cusp} points, respectively. The set of definite fold points and indefinite fold points is denoted by $S_{0}(f)$ and $S_{1}(f)$, respectively. Consequently, to prove Theorem~\ref{main1} and Theorem~\ref{main2}, it suffices to construct a smooth map satisfying these conditions with the described properties.

We here prepare a lemma, which is used in the proof of the theorems. Throughout the following, let $\mathbb{D}^{2}$ denote the $2$-disk.

    \begin{lemma}
    \label{sub1}


Let $\{\psi_t\}_{t \in [0,1]}$ be a smooth family of Morse functions on $\mathbb{D}^2$ satisfying the following. See Figure~\ref{<3} for an example.

\begin{enumerate}
  \item[(i)] the singular points of $\psi_t$ do not depend on $t$, and there are two saddle points $v$ and $v'$ of $\psi_t$ such that the values of $\psi_t$ at the other singular points do not depend on $t$. Moreover, these values are distinct,
  \item[(ii)] $\psi_t(v) > \psi_t(v')$ for $0 \le t < 1/2$, $\psi_{1/2}(v) = \psi_{1/2}(v')$, and $\psi_t(v) < \psi_t(v')$ for $1/2 < t \le 1$,
  \item[(iii)] no other singular value of $\psi_t$ lies between $\psi_t(v)$ and $\psi_t(v')$,
  \item[(iv)] $\psi_t$ has no local minima in the interior of $\mathbb{D}^2$, and $\psi_t$ is constant and takes its minimum value on $\partial\mathbb{D}^2$.
\end{enumerate}

Then, the map $h : \mathbb{D}^2 \times [0,1] \to \mathbb{R} \times [0,1]$ defined by
$h(x,t) = (\psi_t(x), t)$ has no cusp points, and $S_0(h)$ and $S_1(h)$ are trivial braids in
$\mathbb{D}^2 \times [0,1]$. See Figure~\ref{family}. Moreover, $h$ has exactly one singular fiber containing two indefinite fold points, namely the fiber over $(\psi_{1/2}(v), 1/2)$. If $v$ and $v'$ lie in the same connected component of the level set $\psi_{1/2}^{-1}(\psi_{1/2}(v))$, then this fiber is of type $\mathrm{II}^2$, as depicted in Figure~2. Otherwise, $v$ and $v'$ lie in distinct components of the fiber, and the fiber is not of type $\mathrm{II}^2$ or $\mathrm{II}^3$.

    \end{lemma}

\begin{proof}[Proof]


By (i) and (iv), each $\psi_t$ is a Morse function with the same singular points, which are the maxima and the saddle points, so $h$ has no cusp points, and $S_0(h)$ and $S_1(h)$ are of the form $\{\text{maxima}\} \times [0,1]$ and
$\{\text{saddle points}\} \times [0,1]$, respectively. In particular, they are trivial braids. By (i), (ii) and (iii), the singular fiber containing two indefinite fold points appears only at $t = 1/2$, over the singular value $c = \psi_{1/2}(v) = \psi_{1/2}(v')$. If $v$ and $v'$ lie in distinct components of $\psi_{1/2}^{-1}(c)$, then each component of the fiber contains at most one singular point, and hence the fiber is not of type $\mathrm{II}^2$ or $\mathrm{II}^3$. Suppose that $v$ and $v'$ lie in the same component $K$ of $\psi_{1/2}^{-1}(c)$. Since $\psi_{1/2}$ has no local minima in the interior of $\mathbb{D}^2$ by (iv), the fiber is not of type $\mathrm{II}^3$, as shown in Figure~\ref{<3} (middle) and Figure~\ref{family}.

Let $v_1, \ldots, v_n$ be the saddle points of $\psi_t$, as in Figure~\ref{<3}, and fix an integer $i$ with $2 \le i \le n$ such that $v = v_{i-1}$ and $v' = v_i$. We regard $\{ \psi_{t} \}_{ 0 \le t \le 1}$ as an isotopy between Morse functions $\psi_{0}$ and $\psi_{1}$ illustrated in Figure~\ref{<3}. In Figure~\ref{<3}, the isotopy starts at the top corresponding to $\psi_{0}$, goes down to obtain another smooth map $\psi_{1/2} : \mathbb{D}^{2} \to \mathbb{R}$, and goes down to obtain $\psi_{1} : \mathbb{D}^{2} \to \mathbb{R}$. During the isotopy, the singular points of $\psi_t$ do not move, and only the values of $\psi_t$ at $v$ and $v'$ change. At $t=1/2$, the values at $v$ and $v'$ are equal, and for $1/2<t\le 1$, their order is reversed from that for $\psi_0$. For each $t\in[0,1]$, the Reeb graph $\tau_t$ of $\psi_t$ is shown in Figure~\ref{<3} (middle). Here, the \textit{Reeb graph} $\tau_{t}$ of $\psi_{t}$ is defined as the quotient space obtained by contracting each connected component of the level sets of $\psi_{t}$ to a point, which often has the structure of a graph. In Figure~\ref{<3} (left), some singular fibers of $\psi_{t}$, i.e., the preimages of the images of saddle points, are depicted. In Figure~\ref{family} (left), the fat points on $\mathbb{D}^{2} \times \{ 0 \}$, $\mathbb{D}^{2} \times \{ 1/2 \}$, and $\mathbb{D}^{2} \times \{ 1 \}$ represent the intersections of each disk with the set of definite fold points of $h$. In Figure~\ref{family} (left), the strings, whose intersection with each $\mathbb{D}^{2} \times \{ t \}$ consists of two points, represent the set of indefinite fold points of $h$. Thus, the singular set and the singular fibers of the smooth map $h : \mathbb{D}^{2} \times [0,1] \to \mathbb{R} \times [0,1]$ defined by $h(x, t) = (\psi_{t}(x), t)$ are described as shown in Figure~\ref{family}.

    \begin{figure}[htbp]
        \setlength\unitlength{1truecm}
        \begin{picture}(15,20)(0,0)
            \put(0.5,13){\includegraphics[width=0.8\textwidth,clip]{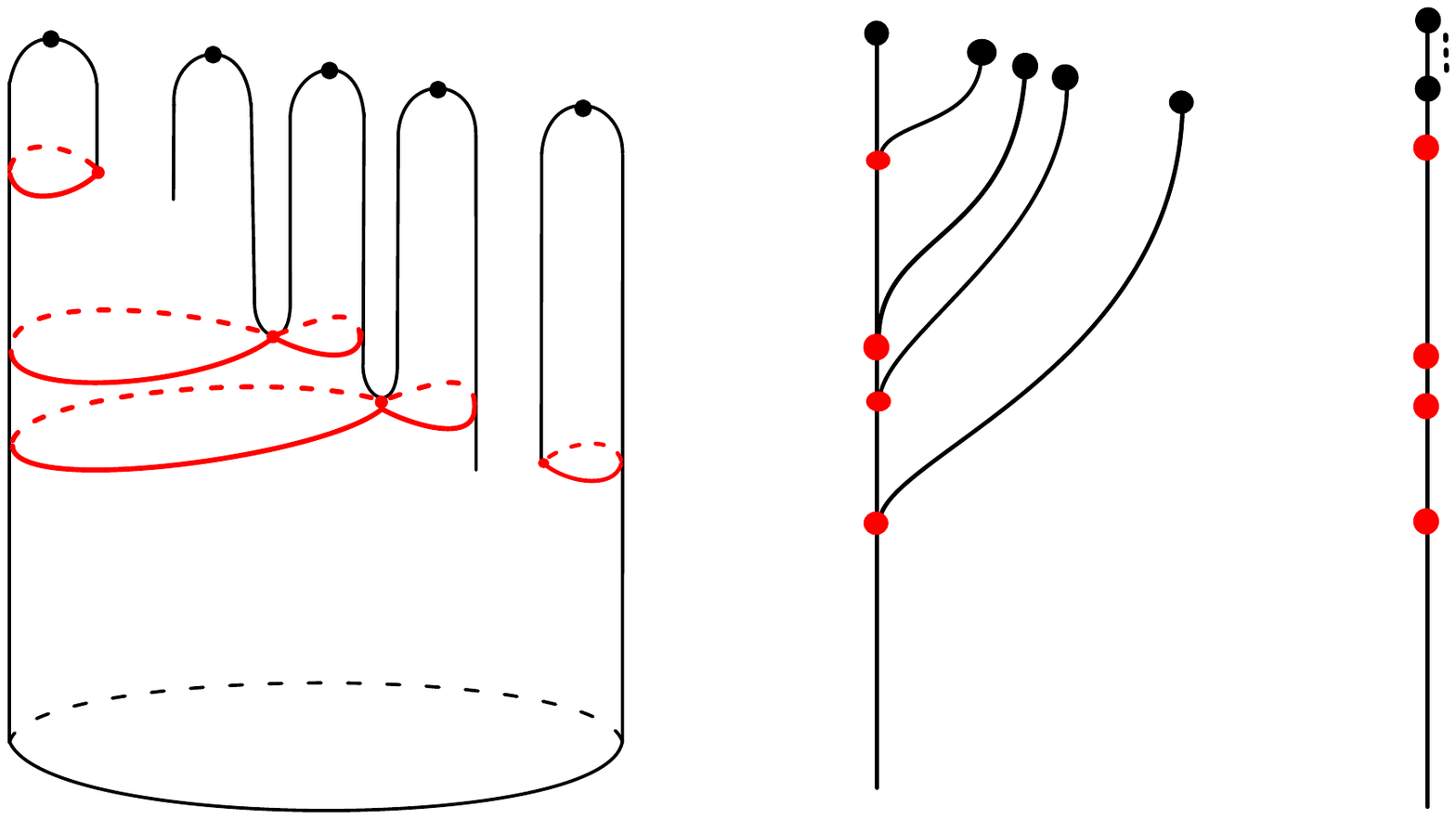}}
                \put(1.2,17.8){$v_{1}$}
                \put(2.2,16.85){$v_{i-1}$}
                \put(3,16.2){$v_{i}$}
                \put(4.2,15.8){$v_{n}$}
            \put(1.3,18.5){$\cdots$}
            \put(3.8,18){$\cdots$}
                    \put(6,18.2){$v_{1}$}
                    \put(5.7,17.1){$v_{i-1}$}
                    \put(6,16.6){$v_{i}$}
                    \put(6,15.8){$v_{n}$}
            \put(6,17.5){$\vdots$}
            \put(6,16.1){$\vdots$}
                    \put(9.7,18.3){$v_{1}$}
                    \put(9.5,17){$v_{i-1}$}
                    \put(9.7,16.6){$v_{i}$}
                    \put(9.7,15.7){$v_{n}$}
            \put(10,17.5){$\vdots$}
            \put(10,16.1){$\vdots$}
                \put(0,16.5){$\psi_{t} :$}
                \put(5.5,13.5){$(0 \le t < \frac{1}{2})$}
                \put(5.5,16.5){$\to$}
                \put(9,16.5){$\to$}

            \put(0.5,6.5){\includegraphics[width=0.8\textwidth,clip]{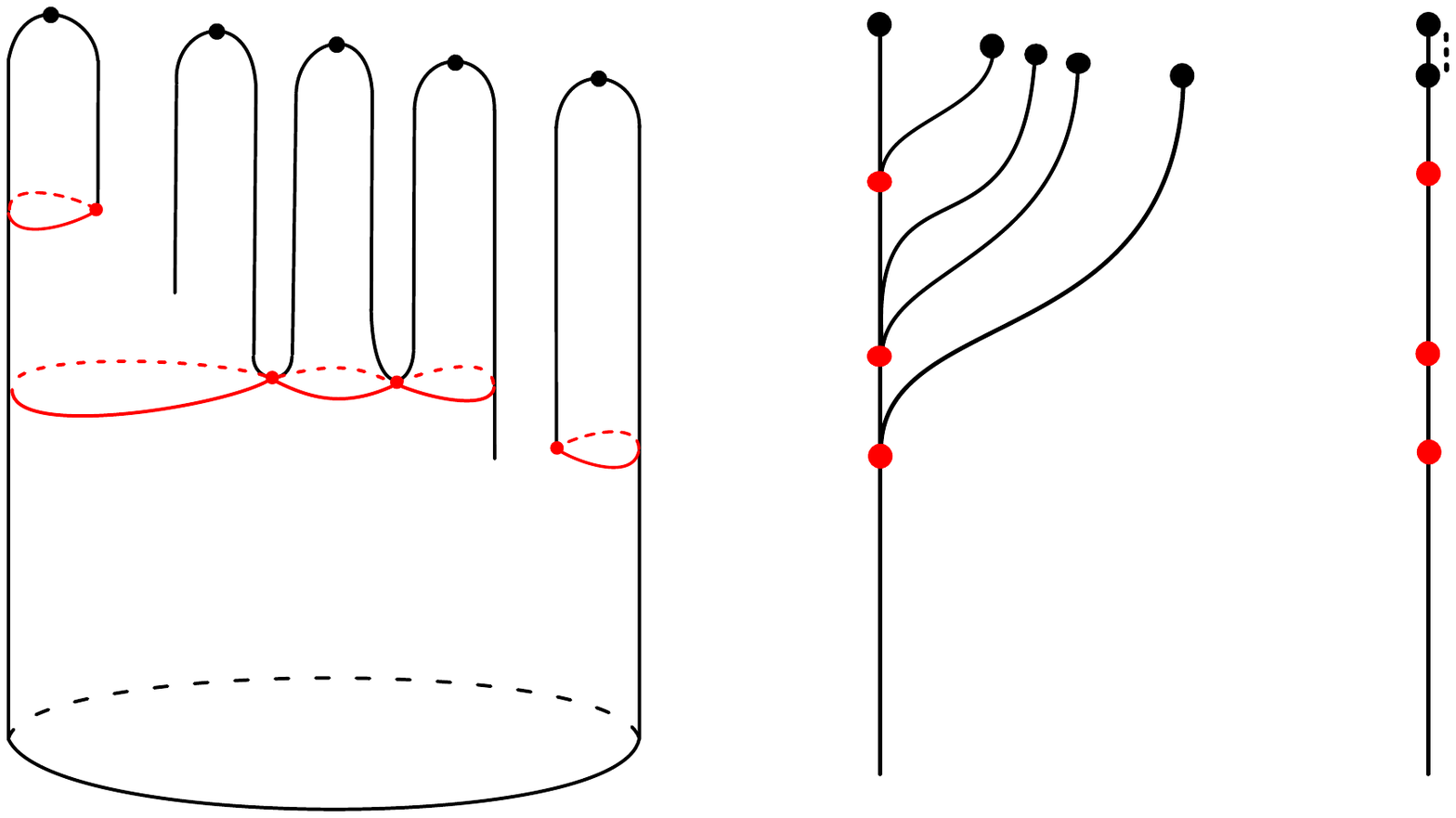}}
                \put(1.2,11){$v_{1}$}
                \put(2,9.7){$v_{i-1}$}
                \put(3.3,9.7){$v_{i}$}
                \put(4.2,9.3){$v_{n}$}
            \put(1.3,12){$\cdots$}
            \put(3.95,12){$\cdots$}
                    \put(6,11.5){$v_{1}$}
                    \put(5.3,10.4){$v_{i-1}, v_{i}$}
                    \put(6,9.6){$v_{n}$}
            \put(6,10.8){$\vdots$}
            \put(6,9.9){$\vdots$}
                    \put(9.7,11.5){$v_{1}$}
                    \put(9,10.5){$v_{i-1}, v_{i}$}
                    \put(9.7,9.6){$v_{n}$}
            \put(10,10.9){$\vdots$}
            \put(10,9.9){$\vdots$}
                \put(-0.2,10){$\psi_{\frac{1}{2}} : $}
                \put(5.5,7){$(t = \frac{1}{2})$}
                \put(5.5,10){$\to$}
                \put(9,10){$\to$}

            \put(0.5,0){\includegraphics[width=0.8\textwidth,clip]{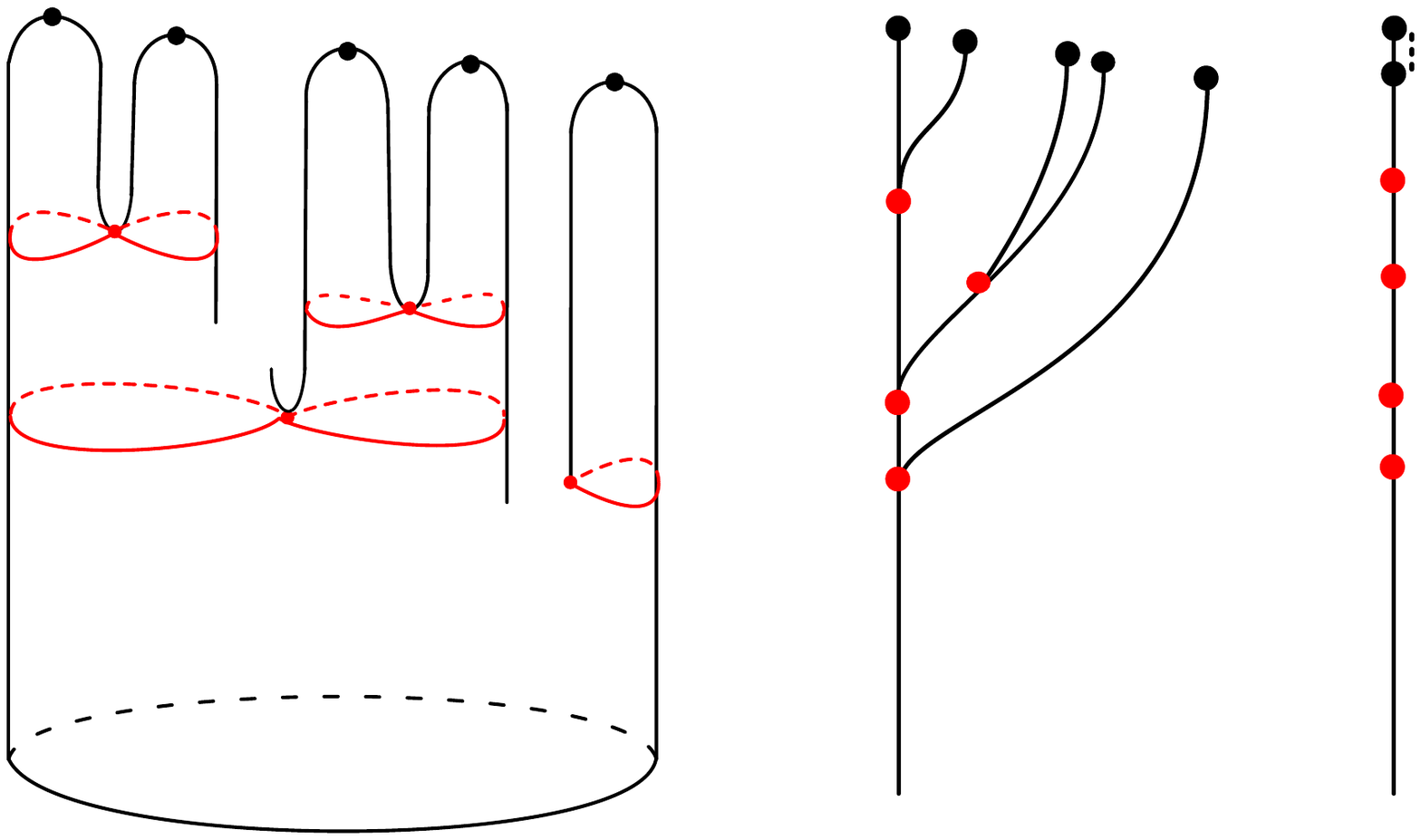}}
                \put(1.2,4.5){$v_{1}$}
                \put(2.3,3.2){$v_{i-1}$}
                \put(3.4,4){$v_{i}$}
                \put(4.5,2.8){$v_{n}$}
            \put(2.3,5){$\cdots$}
            \put(4.2,5){$\cdots$}
                    \put(6.3,5){$v_{1}$}
                    \put(7.4,4.2){$v_{i}$}
                    \put(6.1,3.5){$v_{i-1}$}
                    \put(6.3,3){$v_{n}$}
            \put(7.5,6){$\cdots$}
            \put(8.4,5.8){$\cdots$}
                    \put(9.7,5.2){$v_{1}$}
                    \put(10.5,4.8){$\vdots$}
                    \put(9.7,4.6){$v_{i}$}
                    \put(9.5,3.6){$v_{i-1}$}
                    \put(10.5,3.4){$\vdots$}
                    \put(9.7,3.2){$v_{n}$}
                \put(0,3.6){$\psi_{t}: $}
                \put(5.5,0.2){$(\frac{1}{2} < t < 1)$}
                \put(5.5,3.6){$\to$}
                \put(9,3.6){$\to$}
            
        \end{picture}
        \caption{Morse functions $\psi_{t}$ $(0 \le t \le 1)$.}
        \label{<3}
    \end{figure}

    \begin{figure}[htbp]
        \setlength\unitlength{1truecm}
        \begin{picture}(15,6)(0,0)
            \put(0,-1.5){\includegraphics[width=1\textwidth,clip]{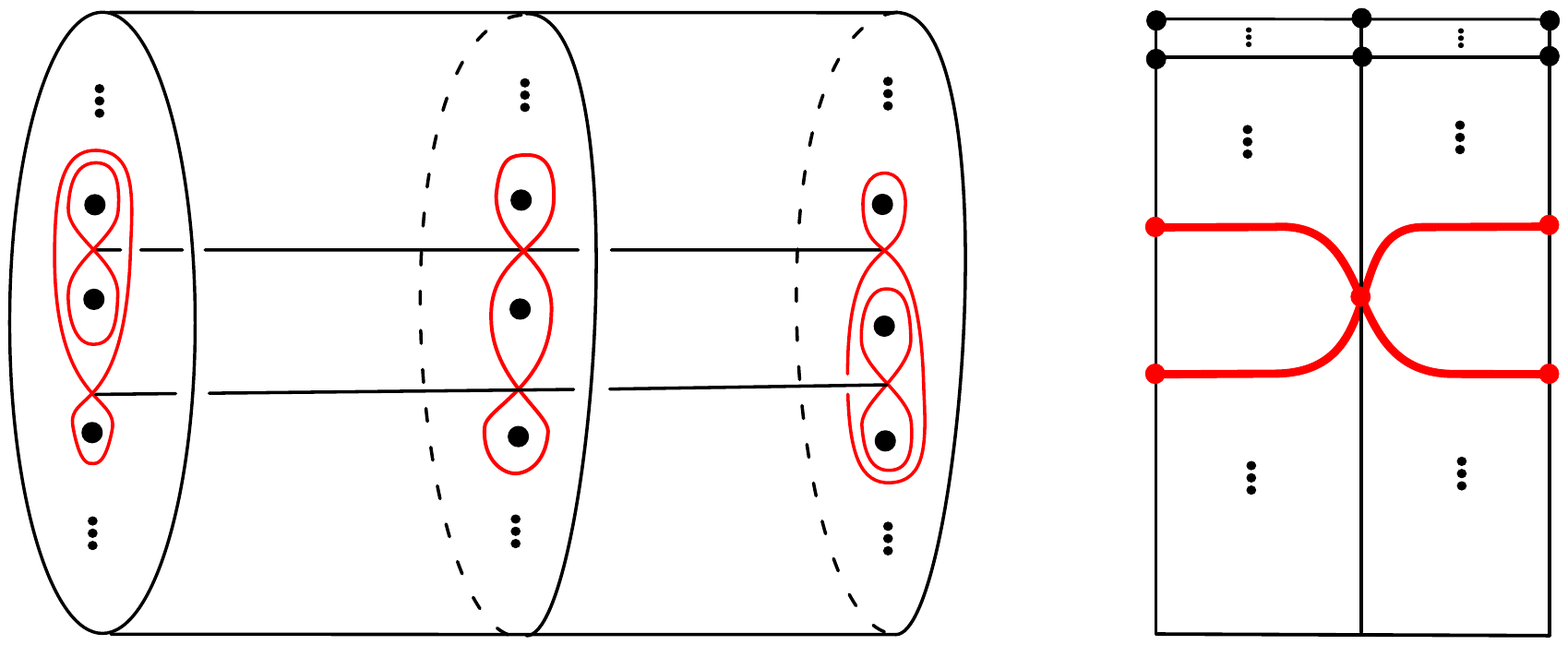}}

            \put(-0.7,3.7){$v_{i-1}$}
            \put(-0.7,2.7){$v_{i}$}

            \put(0,0.2){$\mathbb{D}^{2} \times \{ 0 \}$}
            \put(3.5,0.2){$\mathbb{D}^{2} \times \{ \frac{1}{2} \}$}
            \put(6.5,0.2){$\mathbb{D}^{2} \times \{ 1 \}$}

            \put(8,3.3){$h$}
            \put(8,3){$\longrightarrow$}

            \put(9,0.2){$\{  0 \}$}
            \put(10.5,0.2){$\{  \frac{1}{2} \}$}
            \put(12,0.2){$\{  1 \}$}
        \end{picture}
        \caption{The smooth map $h$.}
        \label{family}
    \end{figure}

\end{proof}

Let us recall some basic properties of links in $S^{3}$. It is well known that any oriented link in the $3$-sphere $S^{3}$ is represented as the closure of a braid, and any braid is presented by a braid word with the generators $\sigma_{i}^{\pm 1}$, which are shown in Figure~\ref{sigma}. See \cite[Chapter 1]{Lickorish-KT}, for example.

    \begin{figure}[htbp]
        {\unitlength=1cm
        \begin{picture}(12.5,4)(0,0)
        \put(3,0){\includegraphics[height=5cm,clip]{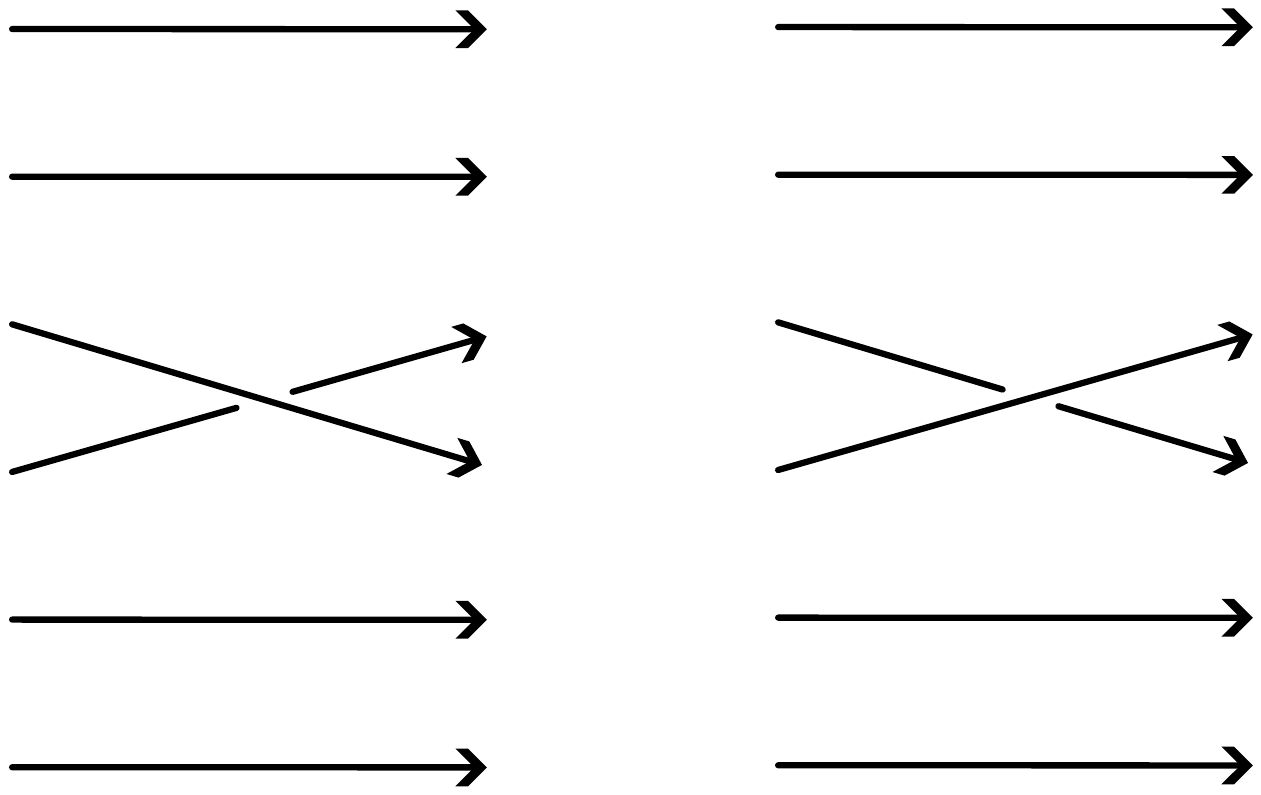}}
        \put(3.5,3.8){$1$}
        \put(3.5,3.2){$2$}
            \put(4.9,2.8){$\vdots$}
        \put(3.5,2.7){$i$}
        \put(3,2){$i+1$}
            \put(4.9,1.8){$\vdots$}
        \put(3,1.4){$n-1$}
        \put(3.5,0.9){$n$}
            \put(8,2.8){$\vdots$}
            \put(8,1.8){$\vdots$}
        
        \put(4.8,0.4){$\sigma_{i}$}
        \put(7.7,0.4){$\sigma_{i}^{-1}$}
        \end{picture}}
        \caption{The generators $\sigma_{i}$ and $\sigma_{i}^{-1}$.}
        \label{sigma}
    \end{figure}

Throughout the following, we denote the boundary of a manifold $N$ by $\partial N$.

\begin{proof}[Proof of Theorem~\ref{main1}.]

Let $L$ be a link in $S^{3}$. We isotope $L$ so that $L$ is represented by the closure of a braid $b$ of $n$ strings $(n \ge 2)$ given as a braid word $\sigma_{z_{1}}^{m_{1}} \sigma_{z_{2}}^{m_{2}} \cdots \sigma_{z_{l}}^{m_{l}}$, where $1 \leq z_{i} \leq n-1$, $z_{i-1} \neq z_{i}$, each $m_{i}$ is a non-zero integer for $i \in \{ 1, 2, \ldots, l \}$, and $l \ge 1$. If $b$ is the trivial braid, we apply a stabilization to $b$ so that the braid word is non-empty.

We will obtain a stable map $S^{3} \to \mathbb{R}^{2}$ by gluing several smooth maps on submanifolds of $S^{3}$. To construct these smooth maps, we make the following preparations. We decompose $S^{3}$ into two solid tori $V_{1}$ and $V_{2}$ as $S^{3} = V_{1} \cup_{\phi} V_{2}$. Here, $\phi$ denotes a diffeomorphism from $\partial V_{1}$ onto $\partial V_{2}$ such that a meridian of $V_{1}$ corresponds to a longitude of $V_{2}$ by $\phi$. (A $\textit{meridian}$ of $V_{1}$ is $\partial \mathbb{D}^{2} \times \{*\} \subset \mathbb{D}^{2} \times S^{1}=V_{1}$ and a $\textit{longitude}$ of $V_{2}$ is $\{ * \} \times S^{1} \subset \mathbb{D}^{2} \times S^{1} = V_{2}$.) We further decompose $V_{1}$ into $N_{1}$ and $N_{2}$, each of which is diffeomorphic to $\mathbb{D}^{2} \times [0,l]$. We further isotope $L$ so that $L$ is contained in $V_{1}$ and $L \cap N_{1}$ corresponds to the braid $b$. Next, we decompose $N_{1}$ into $l$ solid cylinders $W_{1}, \ldots, W_{l}$, where $W_{i}$ is regarded as $\mathbb{D}^{2} \times [i-1,i]$ $(1 \leq i \leq l)$ such that $L \cap W_{i}$ corresponds to $\sigma_{z_{i}}^{m_{i}}$ as shown in Figure~\ref{V1}. Let $A$ be an annular region in $\mathbb{R}^{2}$ and $E$ a disk in $\mathbb{R}^{2}$ such that the intersection $A \cap E = \partial E$. Also, we decompose $A$ into two rectangular regions $R_{1}$ and $R_{2}$ by $\gamma_{1}, \gamma_{l+1}$. Moreover, we decompose $R_{1}$ into $l$ rectangular regions by arcs $\gamma_{2}, \ldots, \gamma_{l}$ as shown in Figure~\ref{A}.

    \begin{figure}[htbp]
        {\unitlength=1cm
    \begin{picture}(10,7.5)(0,0)
        \put(0,0){\includegraphics[height=7cm,clip]{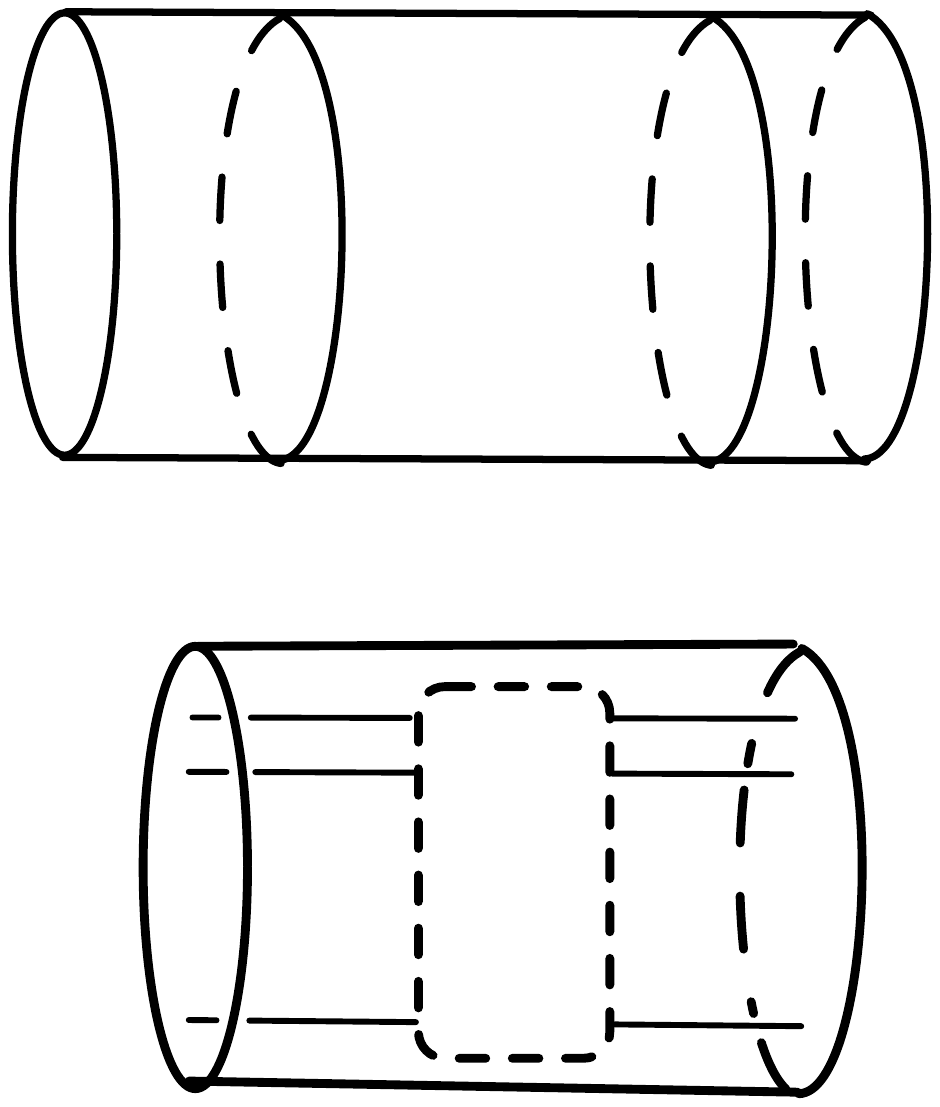}}

        \put(1.5,5.5){$N_{1} :$}

        \put(3.5,3.7){$W_{1}$}
            \put(3,7){$F_{1}$}
            \put(4,7){$F_{2}$}
        \put(5,5.5){$\cdots$}
        \put(7,3.7){$W_{l}$}
            \put(6.5,7){$F_{l}$}
            \put(7.5,7){$F_{l+1}$}

        \put(1.5,2){$W_{i} :$}
        \put(5.3,1.8){$\sigma_{z_{i}}^{m_{i}}$}
        \put(3,3){$1$}
        \put(3,2.5){$2$}
        \put(4.3,1.7){$\vdots$}
        \put(3,1){$n$}
            \put(3.5,0){$F_{i}$}
            \put(7,0){$F_{i+1}$}
        
     \end{picture}}
     \caption{$N_{1} = W_{1} \cup \cdots \cup W_{l}$.}
        \label{V1}
     \end{figure}

\begin{figure}[htbp]
        {\unitlength=1cm
    \begin{picture}(10,6)(0,1)
         \put(2,0){\includegraphics[height=8cm,clip]{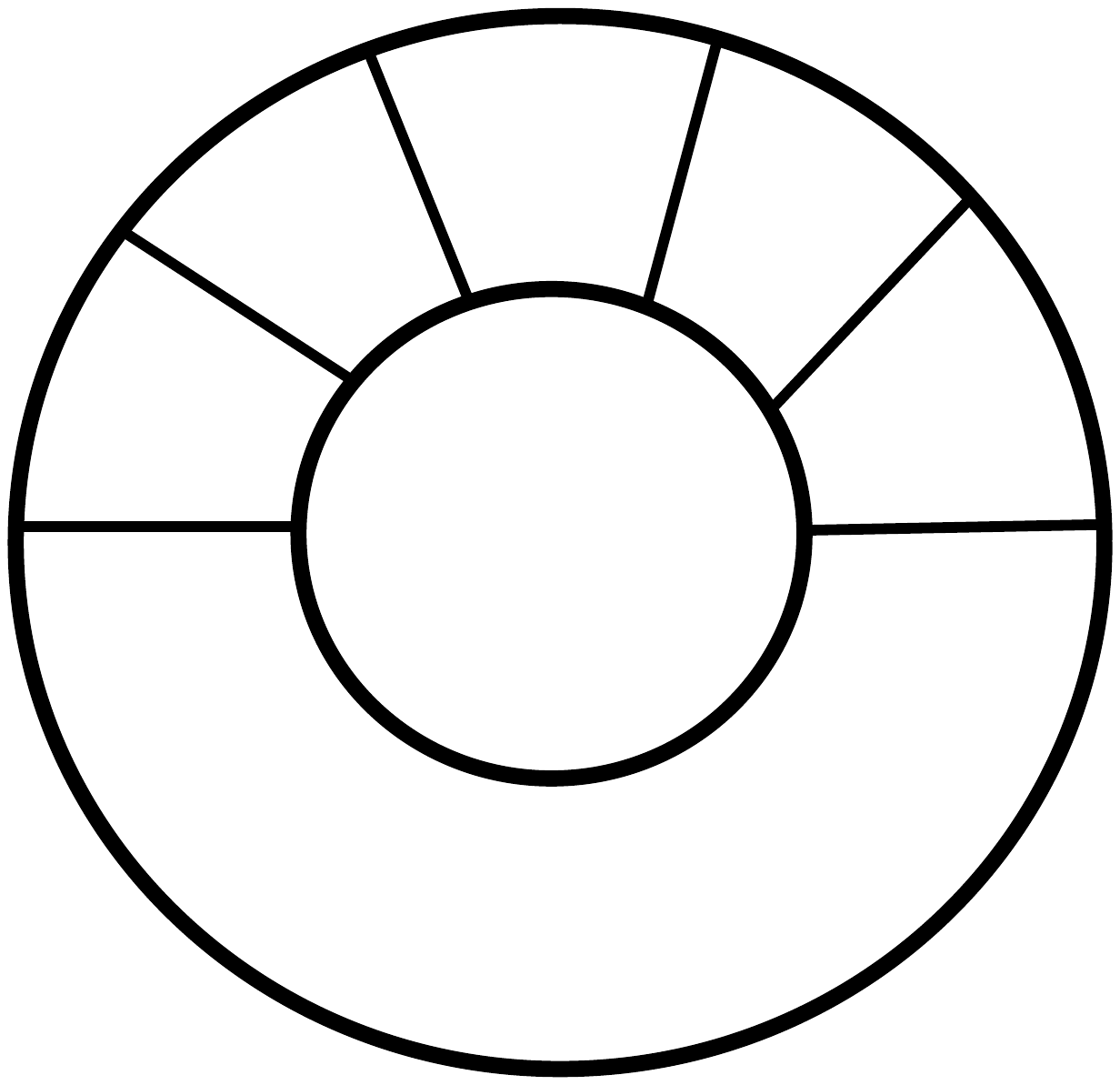}}

    \put(4.7,2.2){$R_{2}$}
         \put(2.7,4.5){$T_{1}$}
            \put(3.5,5.5){$\cdots$}
         \put(4.7,6){$T_{i}$}
            \put(5.7,5.5){$\cdots$}
         \put(6.6,4.5){$T_{l}$}

    \put(1.5,4){$\gamma_{1}$}
    \put(2,5.5){$\gamma_{2}$}
    \put(3.8,6.7){$\gamma_{i}$}
    \put(5.7,6.7){$\gamma_{i+1}$}
    \put(7.3,5.5){$\gamma_{l}$}
    \put(8,4){$\gamma_{l+1}$}

    \end{picture}}
     \caption{$A = R_{1} \cup R_{2}$ and $R_{1} = T_{1} \cup \cdots \cup T_{l}$.}
        \label{A}
\end{figure}

Let $F_{i}$ be the disk $\mathbb{D}^{2} \times \{i-1\}$ in $N_{1}= \mathbb{D}^{2} \times [0, l] = W_{1} \cup \cdots \cup W_{l}$ $(1 \le i \le l+1)$. Note that $F_{i} = W_{i-1} \cap W_{i}$ $(2 \le i \le l)$. For each $i$, let $\psi^{i} : F_{i} \to \gamma_{i}$ be the map which is obtained by regarding $F_{i}$ as $\mathbb{D}^{2}$ and considering a smooth height function on $\mathbb{D}^{2}$ shown in Figure~\ref{tegaki2}. Precisely, $\psi^{i}$ is the composite map of a diffeomorphism $F_{i} \to \mathbb{D}^{2}$ and a smooth height function on $\mathbb{D}^{2}$ shown in Figure~\ref{tegaki2}. In Figure~\ref{tegaki2}, the fat points on $F_{i}$ depict the local maxima of $\psi^{i}$. For the images of saddle points of $\psi^{i}$, their preimages on $F_{i}$ are figure-eight fibers under $\psi^{i}$. The Reeb graph $\tau_{i}$ is shown in the center of Figure~\ref{tegaki2}. The boundary $\partial F_i$ is mapped to an endpoint of $\gamma_i$.

    \begin{figure}[htbp]
        \setlength\unitlength{1truecm}
        \begin{picture}(15,08)(0,0)
            \put(0,0){\includegraphics[width=1\textwidth,clip]{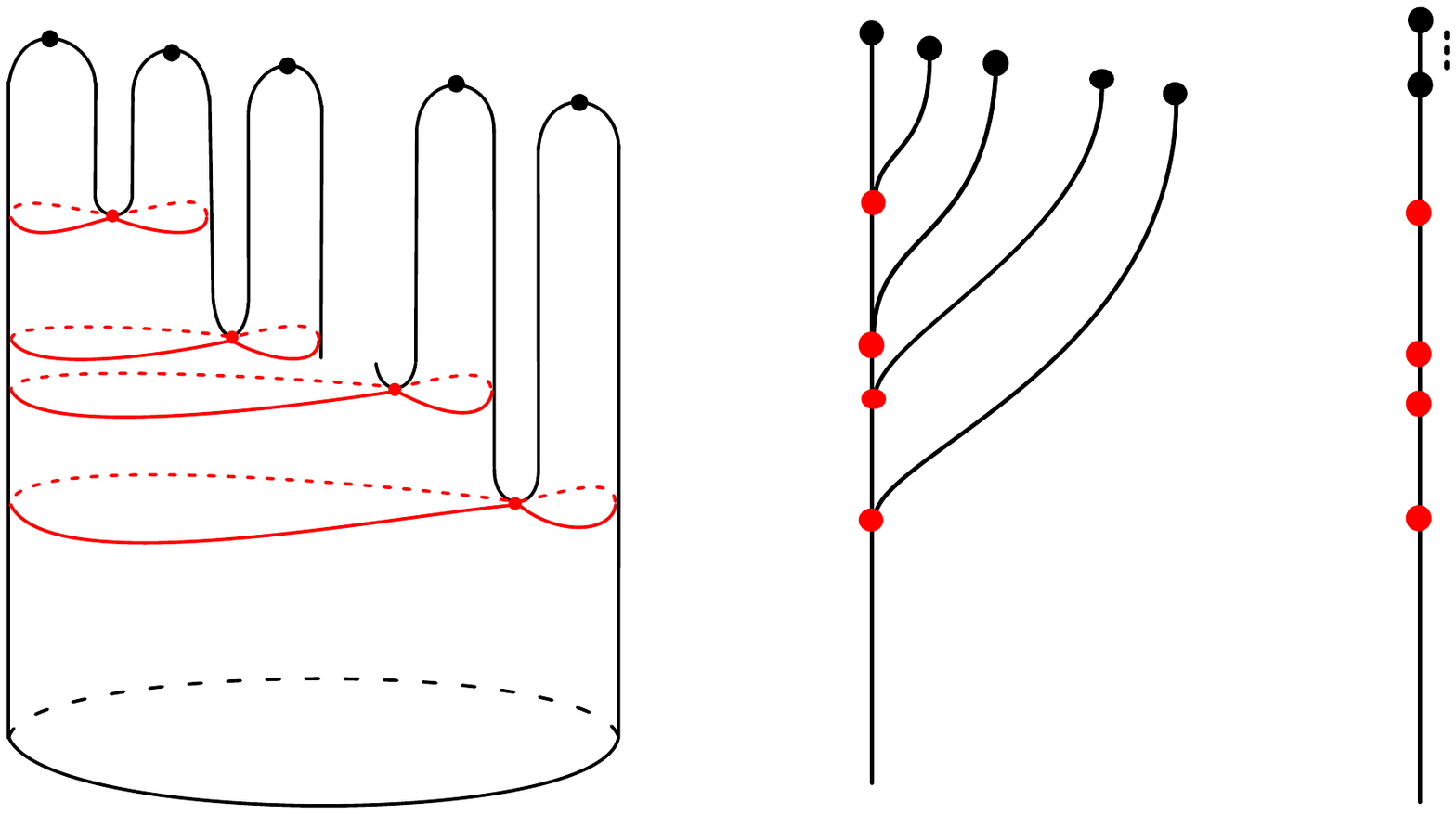}}
            \put(3,6){$\cdots$}
            \put(8.8,7.4){$\cdots$}

        \put(3,0.7){$F_{i}$}
            \put(1,5.7){$v_{1}$}
            \put(2,4.6){$v_{2}$}
            \put(3,4.3){$v_{n-1}$}
            \put(4.3,3.3){$v_{n}$}

                \put(6,4){$\longrightarrow$}
                
        \put(7.5,0.7){$\tau_{i}$}
            \put(7,6.2){$v_{1}$}
            \put(7,5){$v_{2}$}
            \put(6.7,4.5){$v_{n-1}$}
            \put(7,3.5){$v_{n}$}

                \put(10,4){$\longrightarrow$}

        \put(12,0.7){$\gamma_{i}$}
            \put(11,6.2){$\psi^{i}( v_{1} )$}
            \put(11,5){$\psi^{i}( v_{2} )$}
            \put(11.6,4.5){$\vdots$}
            \put(10.8,4.2){$\psi^{i}( v_{n-1} )$}
            \put(11,3.5){$\psi^{i}( v_{n} )$}

            \put(13,6.3){$\alpha$}
            \put(12.5,4.8){ $\stretchleftright{ \} }{\rule{0ex}{90pt}}{.}$ }
                
        \end{picture}
        \caption{A smooth map $\psi^{i} : F_{i} \to \gamma_{i}$ for $1 \leq i \leq l$. The saddle points $v_{1}, \cdots, v_{n}$ in the left correspond to the strands of the braid. See Figure~\ref{sigma}.}
        \label{tegaki2}
    \end{figure}


Consider the case that $z_{i}=1$ and $m_{i}$ is even. A smooth map $\Psi^{i} : F_{i} \times I = W_{i} \to \gamma_{i} \times I = T_{i}$ is obtained from an isotopy between $\psi^{i}$ and $\psi^{i+1}$ as follows. Let $v_{1}, \ldots, v_{n}$ be the saddle points on $F_{i}$ for $\psi^{i}$ and $\alpha$ a left-open interval in $\gamma_{i}$ containing $\psi^{i} (v_{1})$, $\psi^{i}(v_{2})$ and $\max \mathrm{Im} \psi^{i}$. Then, the isotopy $F_{i} \times [0,1] \to F_{i}$ is given by $m_{i}$ half-twists on the connected component of the preimage of $\alpha$ containing $v_{1}, v_{2}$. See Figure~\ref{isoeven1} for the isotopy from $\psi^{i}$ to $\psi^{i+1}$, and see Figure~\ref{even1} for $\Psi^{i}$. Note that the images of definite fold points and indefinite fold points have no normal crossings in this case. In Figure~\ref{even1}, the image of arcs in $W_{i}$ consists of the fat segments in $T_{i}$. The image of the set of definite fold points consists of the parallel segments in $T_{i}$.

    \begin{figure}[htbp]
        \setlength\unitlength{1truecm}
        \begin{picture}(15,7)(0,1)
            \put(0,0){\includegraphics[width=1\textwidth,clip]{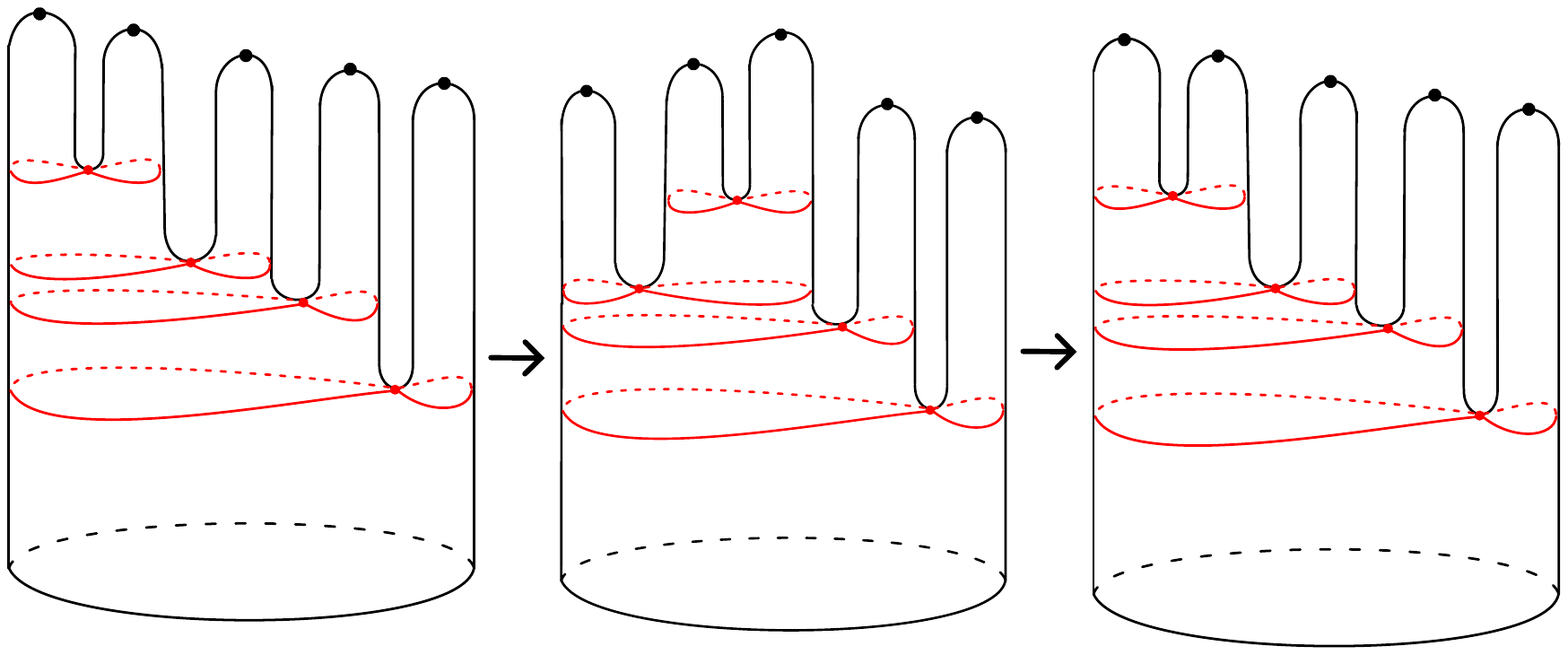}}    
            \put(0.7,5.3){$v_{1}$}
            \put(1.5,4.6){$v_{2}$}
            \put(2.3,4.3){$v_{3}$}
            \put(3.2,3.5){$v_{4}$}

            \put(5.8,5.1){$v_{1}$}
            \put(5,4.4){$v_{2}$}
            \put(6.6,4){$v_{3}$}
            \put(7.4,3.3){$v_{4}$}

            \put(9.4,5.1){$v_{1}$}
            \put(10.3,4.4){$v_{2}$}
            \put(11,4.1){$v_{3}$}
            \put(11.8,3.3){$v_{4}$}
        \end{picture}
        \caption{$m_{i}=2$}
        \label{isoeven1}
    \end{figure}

    \begin{figure}[htbp]
        \setlength\unitlength{1truecm}
        \begin{picture}(15,7.5)(0,1)
            \put(0,0){\includegraphics[width=1\textwidth,clip]{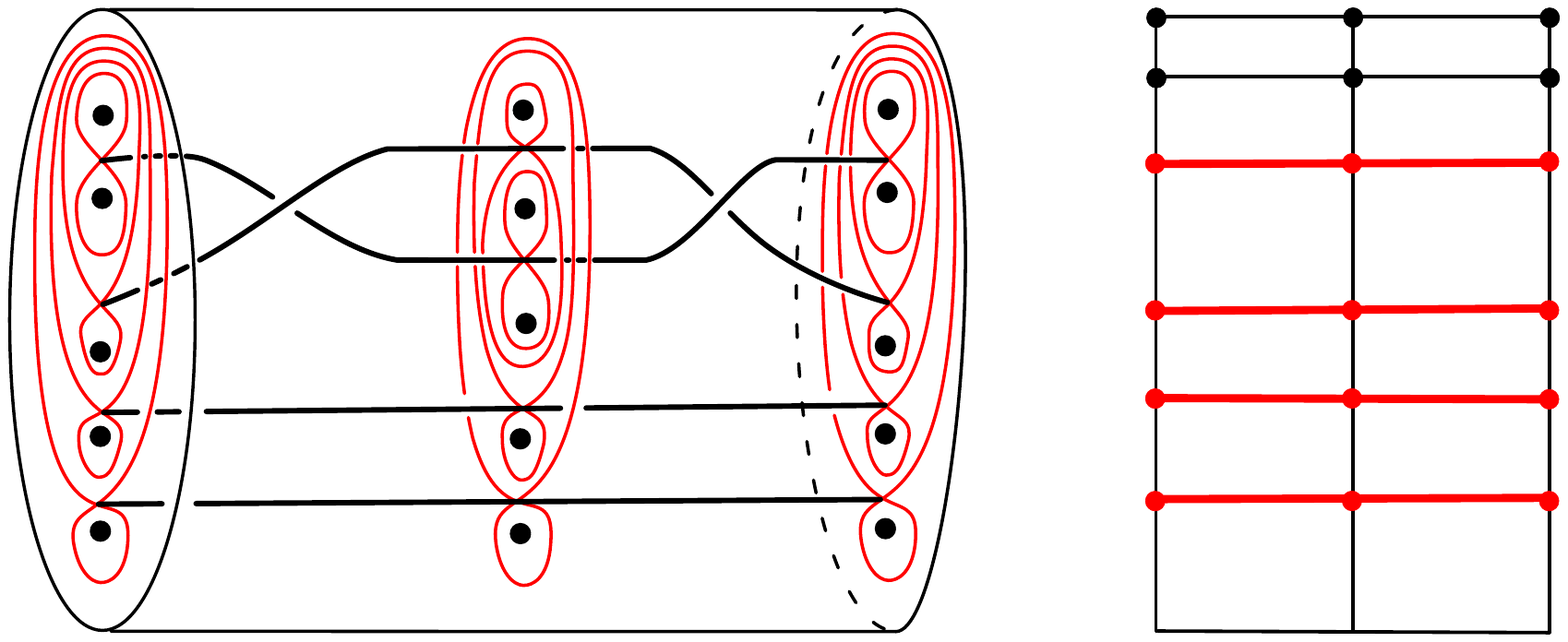}}

            \put(4,1.5){$W_{i}$}
                \put(-0.5,6){$v_{1}$}
                \put(-0.5,5){$v_{2}$}
                \put(-0.5,4){$v_{3}$}
                \put(-0.5,3){$v_{4}$}
            \put(8,5){$\longrightarrow$}
            \put(10.5,1.5){$T_{i}$}
                \put(11.25,6.65){$\vdots$}
                \put(9.6,6.65){$\vdots$}
                
        \end{picture}
        \caption{This case is $n=4$ and $m_{i}=2$.}
        \label{even1}
    \end{figure}


Consider the case that $z_{i}=1$ and $m_{i}$ is odd. A smooth map $\Psi^{i} : F_{i} \times I = W_{i} \to \gamma_{i} \times I = T_{i}$ is obtained by connecting three isotopies as follows. The first isotopy is identical to the one used in the case where $z_{i}=1$ and $m_{i}$ is even. The second isotopy is given by using Lemma~\ref{sub1}. See Figure~\ref{uselemma}. Note that the isotopy switches the heights of saddle points $v_{1}, v_{2}$ on $\mathbb{D}^{2}$. The third isotopy is defined by switching the heights of pairs among the three maxima as shown in Figure~\ref{switch}. The third one induces the switching of the images of the three uppermost maxima of $\mathbb{D}^{2}$ in the target. For the obtained smooth map $\Psi^{i}$ by connecting three isotopies, see Figure~\ref{odd1}. Note that the image of indefinite fold points, which is isotopic to $L \cap W_{i}$, has a normal crossing, and the image of definite fold points has three normal crossings when $m_{i}$ is odd. In Figure~\ref{odd1}, the image of arcs in $W_{i}$ consists of the fat arcs in $T_{i}$. The image of the set of definite fold points is braided with the upper arcs in $T_{i}$.

    \begin{figure}[htbp]
    \setlength\unitlength{1truecm}
        \begin{picture}(15,9)(0,0)
        \centering
        \put(0,0){\includegraphics[width=1\linewidth]{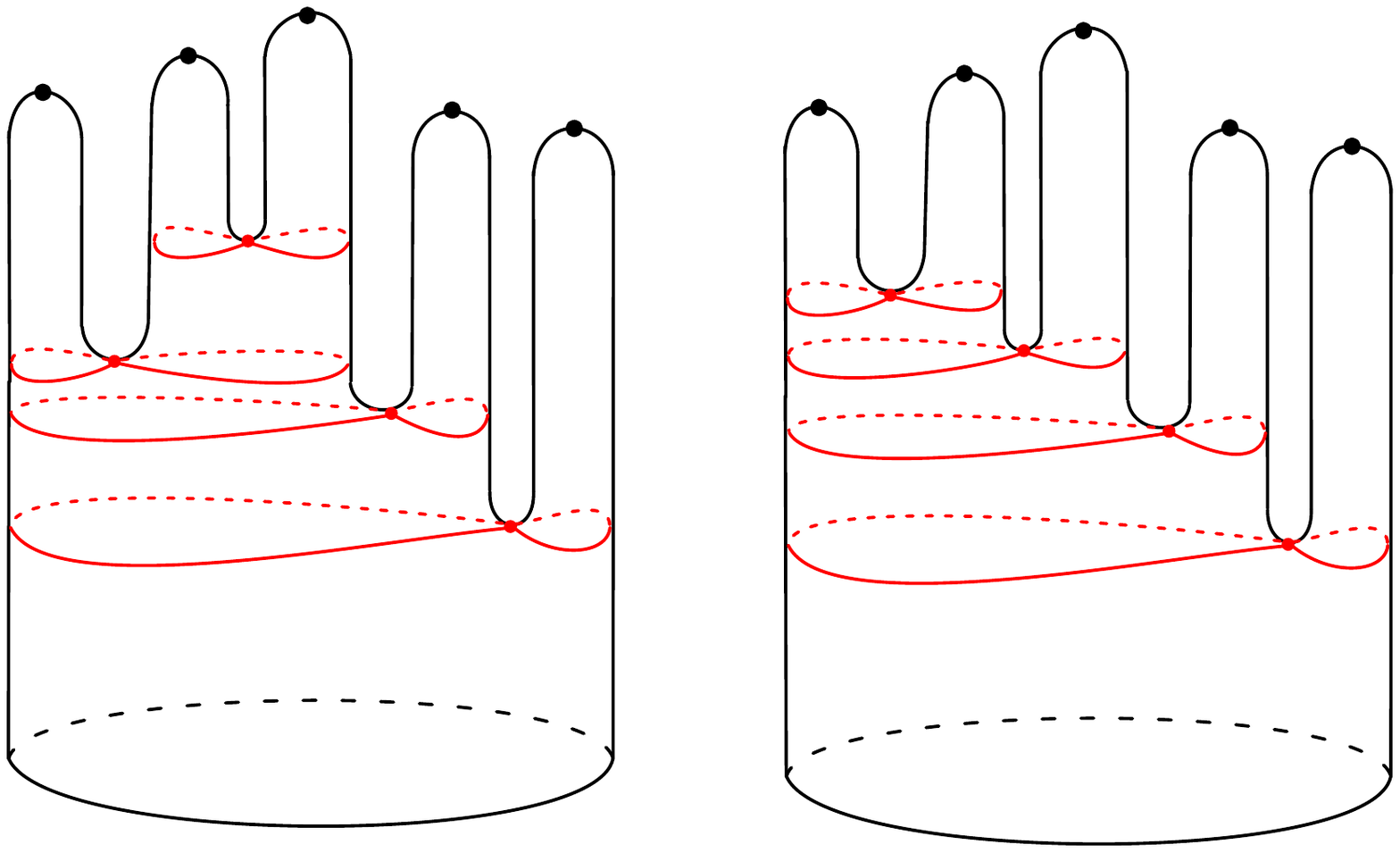}}

        \put(2,5.5){$v_{1}$}
        \put(1,4.5){$v_{2}$}
        \put(3.3,4){$v_{3}$}
        \put(4.2,3){$v_{4}$}
        
        \put(5.8,5){$\longrightarrow$}

        \put(7.8,5.3){$v_{1}$}
        \put(9,4.7){$v_{2}$}
        \put(10,4){$v_{3}$}
        \put(11,3){$v_{4}$}
        
        \end{picture}
        \caption{The isotopy is defined by Lemma~\ref{sub1}.}
        \label{uselemma}
    \end{figure}

    \begin{figure}[htbp]
        \setlength\unitlength{1truecm}
        \begin{picture}(15,9)(0,0)
            \put(0,0){\includegraphics[width=1\textwidth,clip]{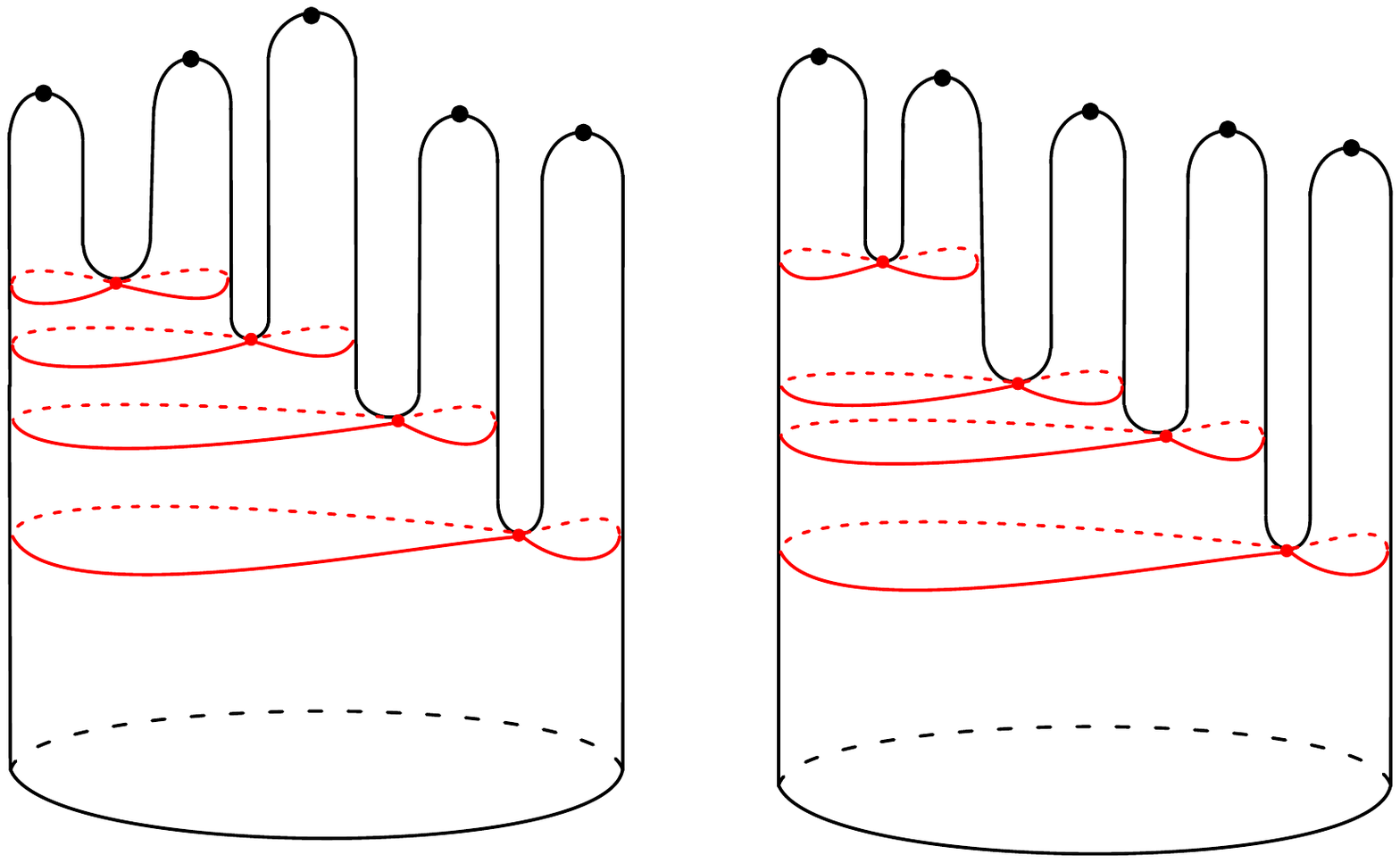}}

            \put(1,5.5){$v_{1}$}
            \put(2,5){$v_{2}$}
            \put(3.3,4){$v_{3}$}
            \put(4.2,3){$v_{4}$}


        \put(5.8,5){$\longrightarrow$}

            \put(7.8,5.3){$v_{1}$}
            \put(9,4.7){$v_{2}$}
            \put(10,4){$v_{3}$}
            \put(11,3){$v_{4}$}
        
            
        \end{picture}
        \caption{The switching of local maxima.}
        \label{switch}
    \end{figure}

    \begin{figure}[htbp]
        \setlength\unitlength{1truecm}
        \begin{picture}(15,7)(0,1)
            \put(0,0){\includegraphics[width=1\textwidth,clip]{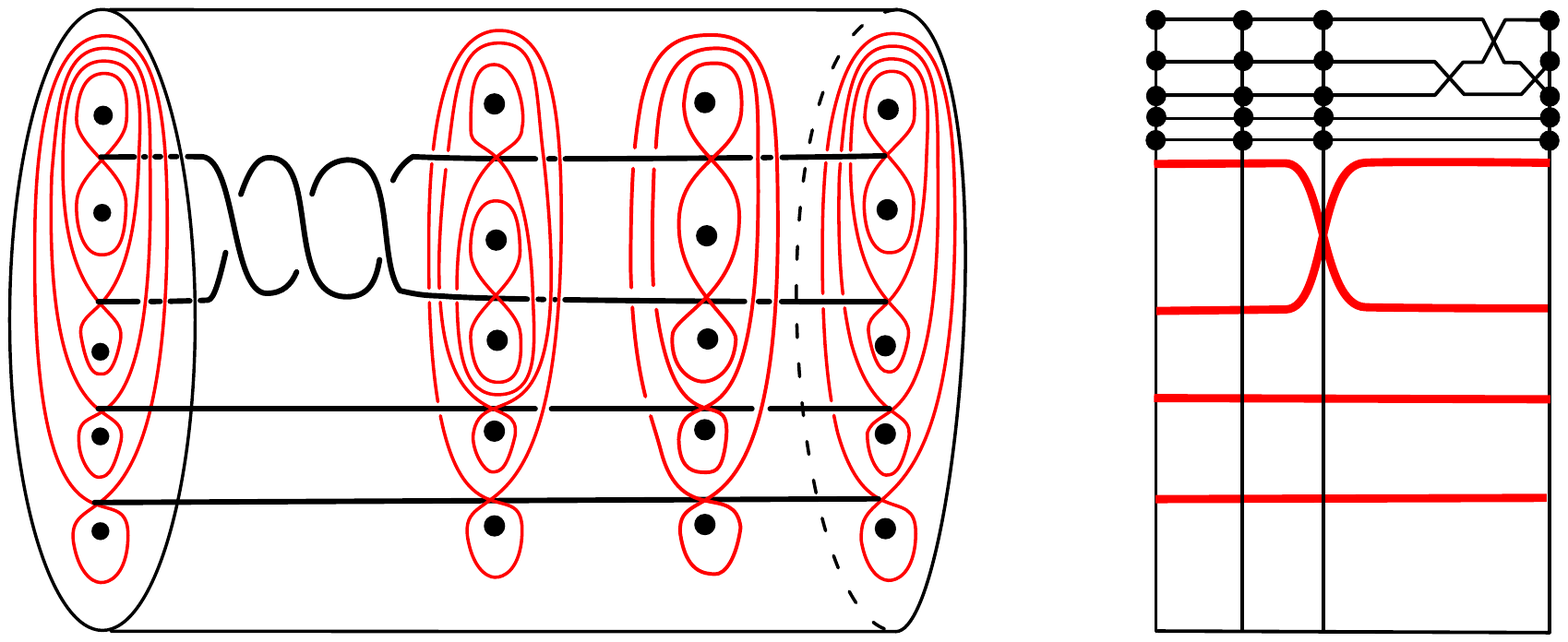}}

            \put(4,1.5){$W_{i}$}
                \put(-0.5,6){$v_{1}$}
                \put(-0.5,5){$v_{2}$}
                \put(-0.5,4){$v_{3}$}
                \put(-0.5,3){$v_{4}$}
            \put(8,5){$\longrightarrow$}
            \put(10.5,1.5){$T_{i}$}
                
        \end{picture}
        \caption{This case is $n=4$, $z_{i}=1$ and $m_{i}=3$.}
        \label{odd1}
    \end{figure}


Consider the case that $z_{i} \neq 1$ and $m_{i}$ is even. A smooth map $\Psi^{i}: F_{i} \times I = W_{i} \to \gamma_{i} \times I = T_{i}$ is obtained by connecting three isotopies as follows. First, we consider an isotopy which reorders the heights of the three saddle points
$v_{z_{i}-1}$, $v_{z_i}$, $v_{z_{i}+1}$ of $\psi^{i}$ as shown in Figure~\ref{reorder}. This is performed by first swapping $v_{z_i}$ and $v_{z_{i}+1}$, then $v_{z_{i}-1}$ and $v_{z_{i}+1}$, and finally $v_{z_{i}-1}$ and $v_{z_i}$. Each of these three steps is an isotopy as in Lemma~\ref{sub1}. In the first and the third steps, the two saddle points lie in the same component of the level set at the moment when their values coincide. Hence each of these steps yields exactly one singular fiber of type $\mathrm{II}^2$. In the second step, $v_{z_{i}-1}$ and $v_{z_{i}+1}$ lie in distinct components of the level set, and hence no singular fiber of type $\mathrm{II}^2$ appears. Therefore, the isotopy in Figure~\ref{reorder} has exactly two singular fibers of type $\mathrm{II}^2$, and so does its inverse. Second, similarly to the case where $z_{i}=1$, we consider an isotopy consisting of $m_{i}$ half-twists on the disk. Let $\beta$ be a left-open interval in $\gamma_{i}$ containing $\psi^{i}(v_{z_{i}})$, $\psi^{i}(v_{z_{i}+1})$ and $\max \mathrm{Im} \psi^{i}$. The isotopy $F_{i} \times [0,1] \to F_{i}$ is given by $m_{i}$ half-twists on the connected component of the preimage of $\beta$ containing $v_{z_{i}}, v_{z_{i}+1}$. Finally, we consider an isotopy which reorders saddle points $ v_{z_{i}-1}, v_{z_{i}}, v_{z_{i}+1}$ again. The inverse of the isotopy shown in Figure~\ref{reorder} restores the three saddle points to their original heights. For the obtained smooth map $\Psi^{i}$ by connecting three isotopies, see Figure~\ref{evenne1}. Note that the image of the indefinite fold points has six normal crossings and the image of definite fold points has no normal crossings. Also, the smooth map $ \Psi^{i} : W_{i} \to T_{i}$ has four singular fibers of type $\mathrm{I\hspace{-1.2pt}I^{2}}$.

    \begin{figure}[htbp]
        \setlength\unitlength{1truecm}
        \begin{picture}(15,14)(0,0)
            \put(-1.2,1){\includegraphics[width=1.3\textwidth,clip]{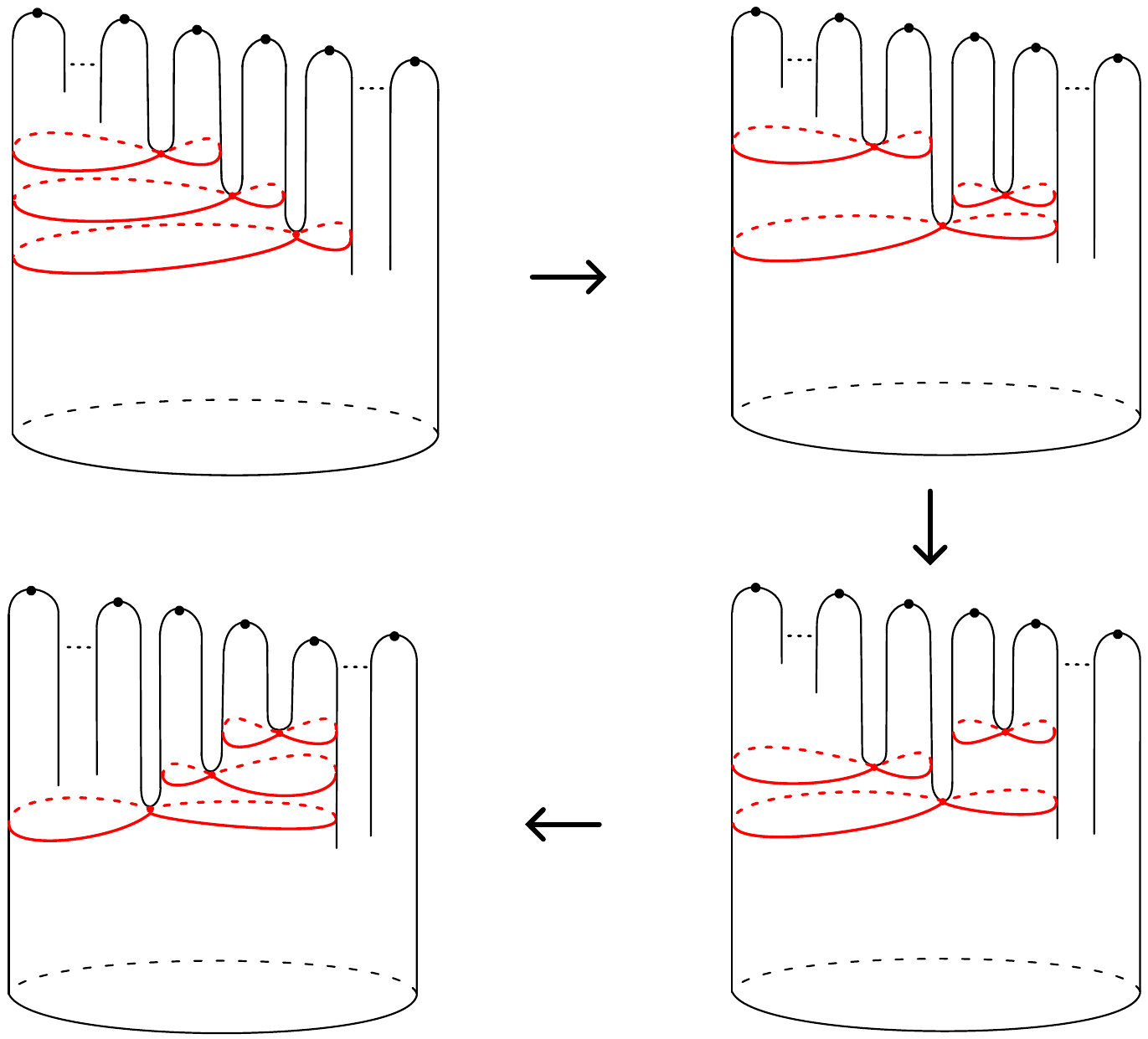}}

            \put(1.5,10.5){$v_{z_{i}-1}$}
            \put(2.5,10){$v_{z_{i}}$}
            \put(3,9.5){$v_{z_{i}+1}$}

            \put(9.5,10.5){$v_{z_{i}-1}$}
            \put(10.3,9.7){$v_{z_{i}}$}
            \put(11.1,10){$v_{z_{i}+1}$}

            \put(9.5,3.6){$v_{z_{i}-1}$}
            \put(10.4,3.2){$v_{z_{i}}$}
            \put(11,4){$v_{z_{i}+1}$}

            \put(1.5,3){$v_{z_{i}-1}$}
            \put(2.3,3.5){$v_{z_{i}}$}
            \put(3,4){$v_{z_{i}+1}$}
           
        \end{picture}
        \caption{Saddle points $v_{z_{i}-1}, v_{z_{i}}, v_{z_{i}+1}$ are reordered.}
        \label{reorder}
    \end{figure}

    \begin{figure}[htbp]
        \setlength\unitlength{1truecm}
        \begin{picture}(15,15)(0,1)
            \put(0,7){\includegraphics[width=1\textwidth,clip]{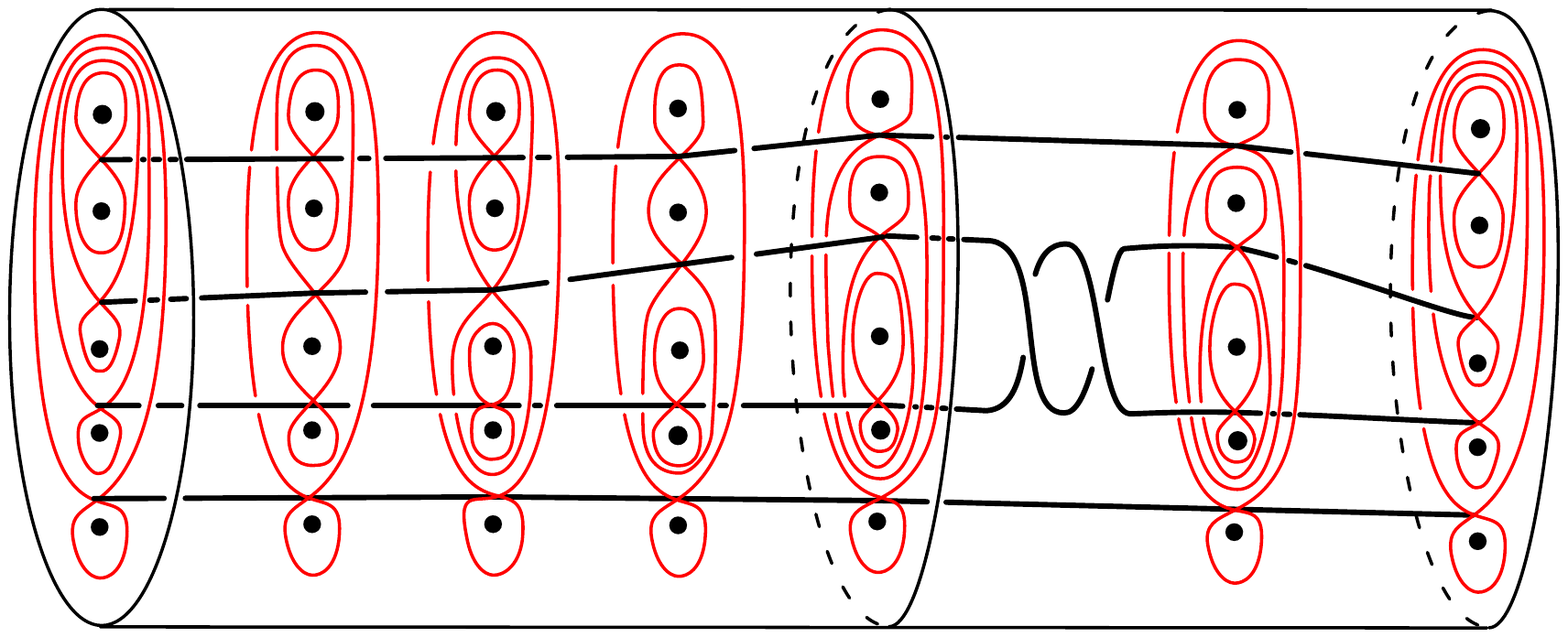}}
            \put(5,8){$W_{i}$}
                \put(-0.5,13){$v_{1}$}
                \put(-0.5,12){$v_{2}$}
                \put(-0.5,11){$v_{3}$}
                \put(-0.5,10){$v_{4}$}
                \put(3,4){$\longrightarrow$}
            \put(2,0){\includegraphics[width=1\textwidth,clip]{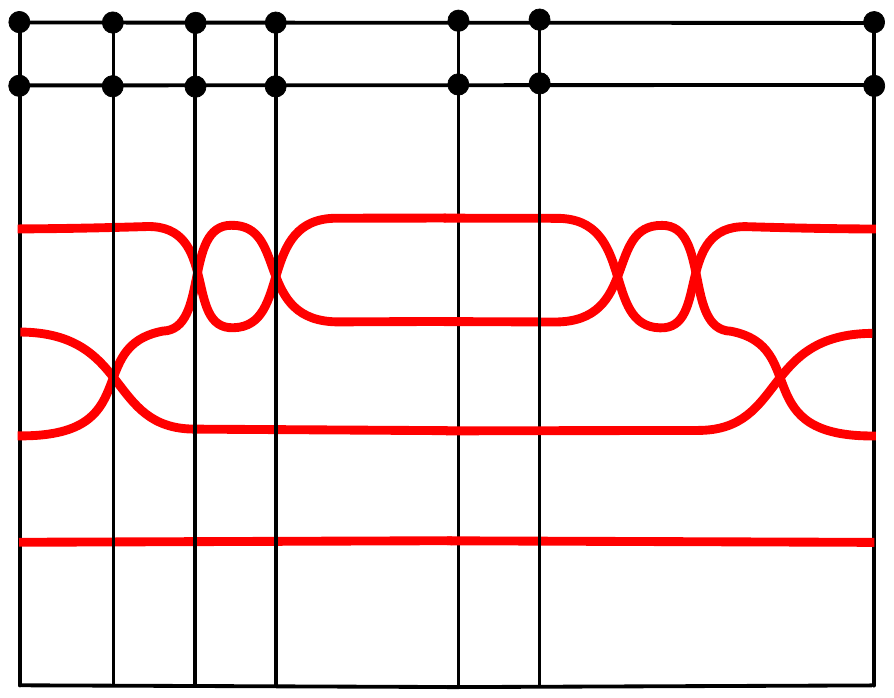}}

            \put(10.5,6.35){$\vdots$}

        \end{picture}
        \caption{This case is $n=4$, $z_{i} = 2$ and $m_{i}=2$.}
        \label{evenne1}
    \end{figure}


Consider the case that $z_{i} \neq 1$ and $m_{i}$ is odd. A smooth map $\Psi^{i} : F_{i} \times I = W_{i} \to \gamma_{i} \times I = T_{i}$ is obtained by connecting four isotopies as follows. The first isotopy is identical to the one used in the case where $z_{i} \ne 1$ and $m_{i}$ is even. See Figure~\ref{reorder}. The second isotopy given by $m_{i}$ half-twists is identical to the one used in the case where $z_{i} \ne 1$ and $m_{i}$ is even. The third isotopy is an isotopy given by using Lemma~\ref{sub1}, which reorders three saddle points $v_{z_{i}-1}, v_{z_{i}}, v_{z_{i}+1}$ of $\psi^{i}$. See Figure~\ref{reorderne1}. The fourth isotopy is given by switching the height of the definite fold points as shown in Figure~\ref{swichne1}. For the obtained smooth map $\Psi^{i}$ by connecting four isotopies, see Figure~\ref{oddne1}. Note that the image of the definite fold points has three normal crossings, and the image of indefinite fold points has five normal crossings. The smooth map $\Psi^{i} : W_{i} \to T_{i}$ has four type $\mathrm{I\hspace{-1.2pt}I^{2}}$ singular fibers.

    \begin{figure}[htbp]
        \setlength\unitlength{1truecm}
        \begin{picture}(15,6)(0,0)
            \put(0.3,-2){\includegraphics[width=1\textwidth,clip]{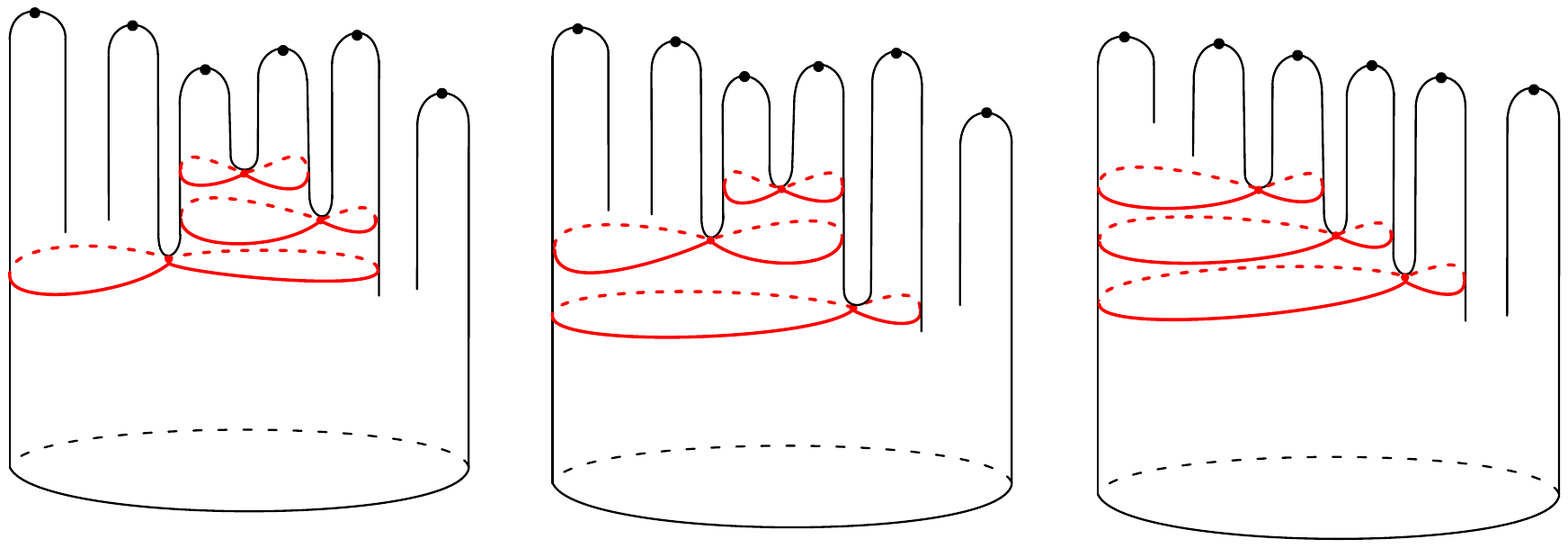}}

            \put(1.5,2.5){$v_{z_{i}-1}$}
            \put(2.8,2.65){$v_{z_{i}}$}
            \put(2,3.1){$v_{z_{i}+1}$}

            \put(4.2,2.5){$\rightarrow$}

            \put(5.8,2.6){$v_{z_{i}-1}$}
            \put(7,2){$v_{z_{i}}$}
            \put(6.4,3){$v_{z_{i}+1}$}

            \put(8.6,2.5){$\rightarrow$}

            \put(10,2.9){$v_{z_{i}-1}$}
            \put(11.5,2.2){$v_{z_{i}}$}
            \put(10.8,2.5){$v_{z_{i}+1}$}
            
        \end{picture}
        \caption{Reordering saddle points $v_{z_{i}-1}, v_{z_{i}}, v_{z_{i}+1}$.}
        \label{reorderne1}
    \end{figure}

    \begin{figure}[htbp]
        \setlength\unitlength{1truecm}
        \begin{picture}(15,8.5)(0,0)
            \put(0,0){\includegraphics[width=1\textwidth,clip]{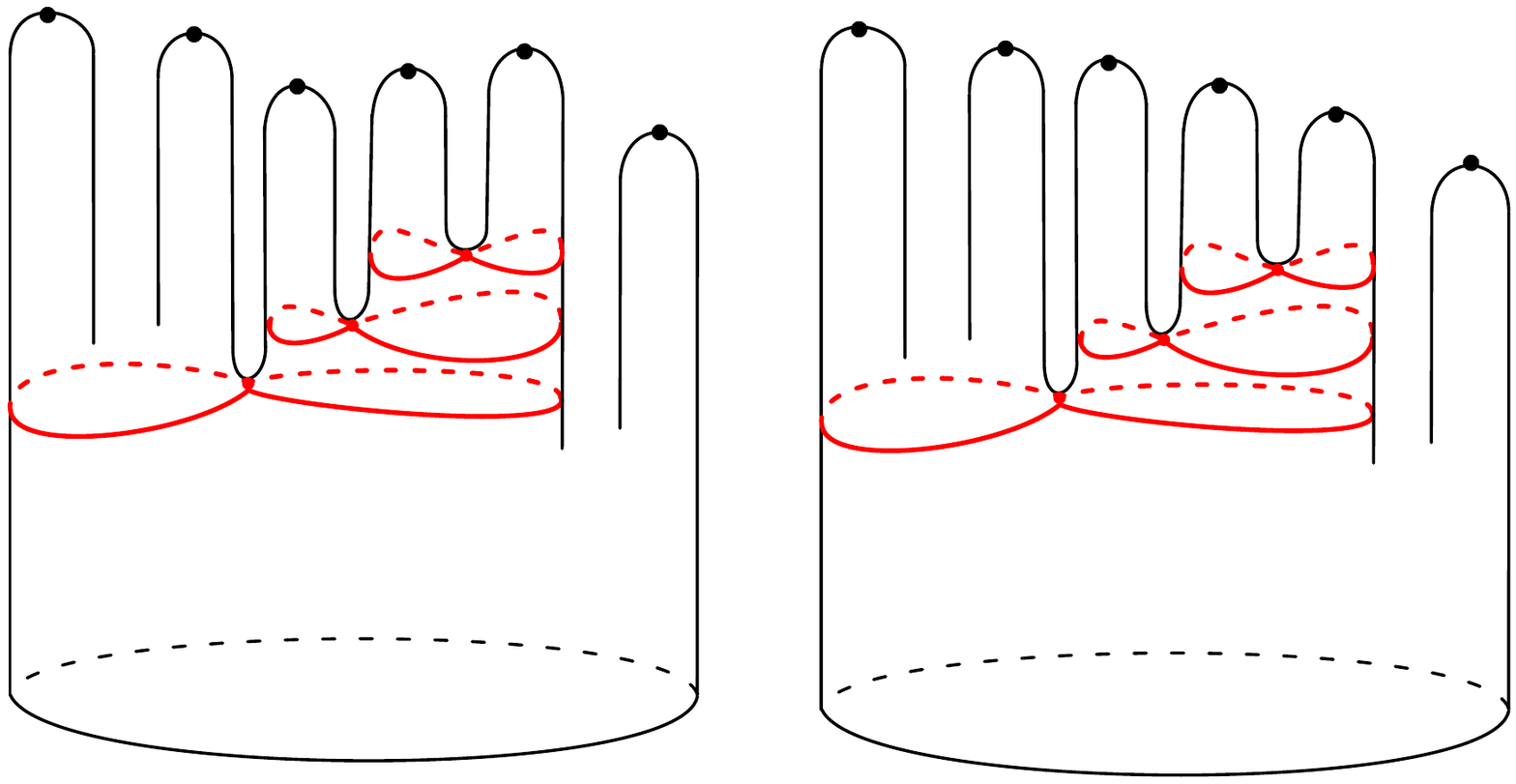}}

            \put(2,3.7){$v_{z_{i}-1}$}
            \put(3.5,4.7){$v_{z_{i}}$}
            \put(3,4.2){$v_{z_{i}+1}$}

            \put(6,4.5){$\longrightarrow$}

            \put(8.5,3.7){$v_{z_{i}-1}$}
            \put(10.5,4.7){$v_{z_{i}}$}
            \put(9.7,4.2){$v_{z_{i}+1}$}
            
        \end{picture}
        \caption{Switching the definite fold points.}
        \label{swichne1}
    \end{figure}

    \begin{figure}[htbp]
        \setlength\unitlength{1truecm}
        \begin{picture}(15,12)(0,0)
            \put(0,5){\includegraphics[width=1\textwidth,clip]{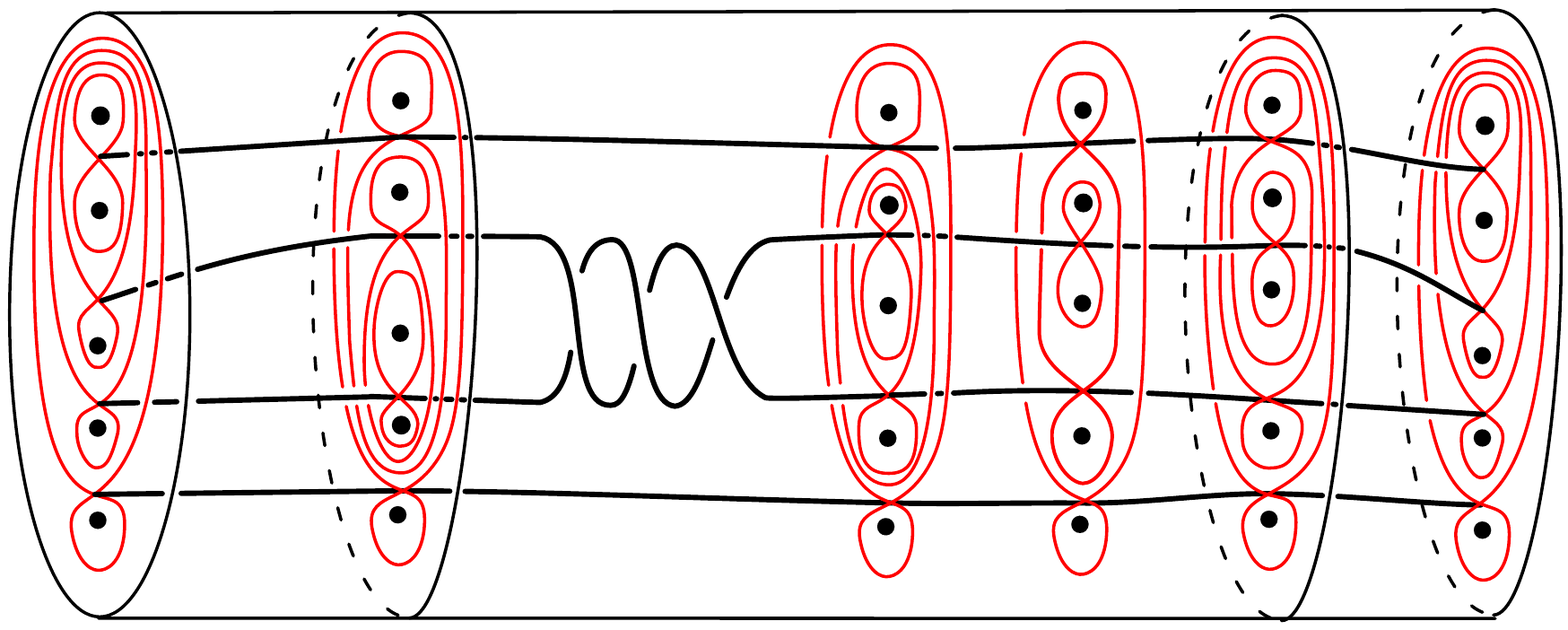}}
            \put(5,6.5){$W_{i}$}
                \put(-0.5,11){$v_{1}$}
                \put(-0.5,10){$v_{2}$}
                \put(-0.5,9){$v_{3}$}
                \put(-0.5,8){$v_{4}$}
                \put(3,3){$\longrightarrow$}
            \put(2,-1){\includegraphics[width=1\textwidth,clip]{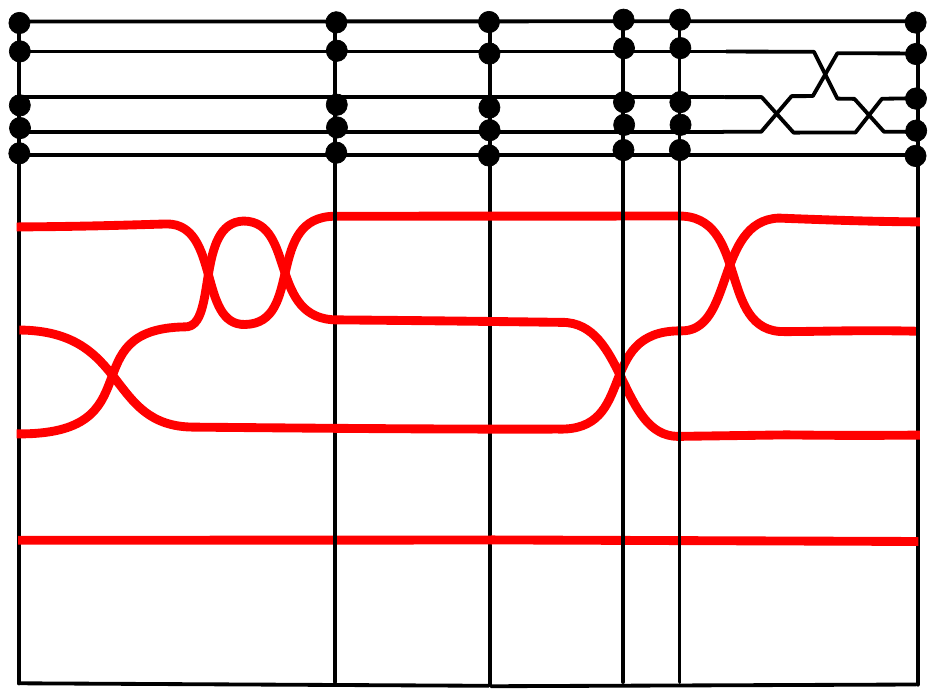}}
        \end{picture}
        \caption{This case is $n=4$, $z_{i} = 2 $ and $m_{i}=3$.}
        \label{oddne1}
    \end{figure}


Here, we may assume that each of the isotopies above is constant near its end points, reparametrizing $t$ if necessary. Then, by connecting $\Psi^{1} : W_{1} \to T_{1}, \ldots, \Psi^{l} : W_{l} \to T_{l}$ along $\psi^{2} : F_{2} \to \gamma_{2}, \ldots, \psi^{l} : F_{l} \to \gamma_{l}$, we obtain a smooth map from the solid cylinder $N_{1}$ into the region $R_{1}$. We connect this map with a smooth map from a solid cylinder $N_{2}$ into the region $R_{2}$ naturally induced from $\psi^{1} : F_{1} \to \gamma_{1}$ and $\psi^{l+1} : F_{l+1} \to \gamma_{l+1}$ by the product structure of $N_{2}$. As a result, a smooth map $g_{1} : V_{1} \to A$ is obtained, where $A$ is an annulus. Further, connecting $g_{1}$ and the natural projection $g_{2} : V_{2} \to E$, we finally obtain a smooth map $f : S^{3} \to A \cup E \subset \mathbb{R}^{2}$. By construction, this map $f$ has no cusp points, and its singular set consists of definite and indefinite fold points. Also, $f$ satisfies the global conditions (5) and (6). Moreover, the set of indefinite fold points forms a link isotopic to $L$ and $f$ has no singular fibers of type $\mathrm{I\hspace{-1.2pt}I^{3}}$. Consequently, $f : S^{3} \to \mathbb{R}^{2}$ is a stable map without cusp points such that $f$ has no singular fibers of type $\mathrm{I\hspace{-1.2pt}I^{3}}$ and $S_{1}(f)$ is isotopic to $L$.

Using the above construction, we next count the singular fibers of type $\mathrm{I\hspace{-1.2pt}I^{2}}$. Suppose that $L$ is represented as the closure of an $n$-braid $\sigma_{z_{1}}^{m_{1}} \sigma_{z_{2}}^{m_{2}} \cdots \sigma_{z_{l}}^{m_{l}}$, where $n \ge 2$, $1 \le z_{i} \le n-1$, $z_{i-1} \neq z_{i}$ for $2 \le i \le l$, and each $m_{i}$ is a non-zero integer. Note that each $\Psi^{i}$ is not a stable map by itself, although $f$ is stable. We mean a singular fiber of type $\mathrm{I\hspace{-1.2pt}I^{2}}$ of $f$ appearing in $\Psi^{i}$ by one of $\Psi^{i}$, and denote the number of such fibers by $|\mathrm{I\hspace{-1.2pt}I^{2}}(\Psi^{i})|$. Then we have $|\mathrm{I\hspace{-1.2pt}I^{2}}(f)| = \sum_{i=1}^{l} |\mathrm{I\hspace{-1.2pt}I^{2}}(\Psi^{i})|$. For each $\sigma_{z_{i}}^{m_{i}}$, we count the singular fibers of type $\mathrm{I\hspace{-1.2pt}I^{2}}$ of $\Psi^{i}$. If $z_{i} = 1$ and $m_{i}$ is even, then $|\mathrm{I\hspace{-1.2pt}I^{2}}(\Psi^{i})| = 0$ (see Figure~\ref{even1}). If $z_{i} = 1$ and $m_{i}$ is odd, then $|\mathrm{I\hspace{-1.2pt}I^{2}}(\Psi^{i})| = 1$.
If $z_{i} \neq 1$ and $m_{i}$ is even, then $|\mathrm{I\hspace{-1.2pt}I^{2}}(\Psi^{i})| = 4$ (see Figure~\ref{evenne1}). If $z_{i} \neq 1$ and $m_{i}$ is odd, then $|\mathrm{I\hspace{-1.2pt}I^{2}}(\Psi^{i})| = 4$ (see Figure~\ref{oddne1}).
Let $X := \{\, i \mid 1 \le i \le l,\ z_{i} = 1 \,\}$. Then

\[
|\mathrm{I\hspace{-1.2pt}I^{2}}(f)| \le 4(l - |X|) + |X| = 4l - 3|X|.
\]

\end{proof}

\section{The proof of Theorem~\ref{main2}}
\label{s3}



From the construction of the stable map in Theorem~\ref{main1}, we obtain Lemma~\ref{sub2}.

\begin{lemma}
    \label{sub2}

    Let $W$ be a solid cylinder $\mathbb{D}^2 \times [0,1]$ and $b_1 \in B_n$ any $n$-braid.
    Then there exists a smooth map $h_1$ from $W$ to $\mathbb{R}^2$ satisfying the following
    conditions.
    \begin{enumerate}
      \item[(i)] $h_1$ has no cusp points and no singular fibers of type $\mathrm{II}^3$
      \item[(ii)] $h_1|_{\mathbb{D}^2 \times \{\varepsilon\}} = \psi$ for $\varepsilon = 0, 1$,
      where $\psi$ is the Morse function with $n$ saddle points shown in Figure~\ref{tegaki2}
      \item[(iii)] $S_1(h_1)$ is an $n$-braid representing $b_1 \in B_n$, that is, it is isotopic to $b_1$ relative to $\mathbb{D}^2 \times \{0,1\}$, and $S_0(h_1)$ is an $(n+1)$-braid
      \item[(iv)] $S_0(h_1) \cap (\mathbb{D}^2 \times \{\varepsilon\})$ and $S_1(h_1) \cap (\mathbb{D}^2 \times \{\varepsilon\})$ consist of the local maxima and the saddle points of $\psi$, respectively.
    \end{enumerate}
\end{lemma}

This lemma means that we can construct a smooth map $h_{1}$ whose set of indefinite fold points is isotopic to a given braid in $W$. However, the set of definite fold points of $h_{1}$ is not controlled in the construction of $h_{1}$. See Figure~\ref{mu}.

\begin{proof}
    In the proof of Theorem~\ref{main1}, following the construction to obtain $N_{1} \to R_{1}$, we get a desired smooth map, where the notation is as in the proof of Theorem~\ref{main1}.
\end{proof}



                
            

            

    \begin{figure}[htbp]
        \setlength\unitlength{1truecm}
        \begin{picture}(15,7)(0,1.5)
            \put(0,0){\includegraphics[width=1\textwidth,clip]{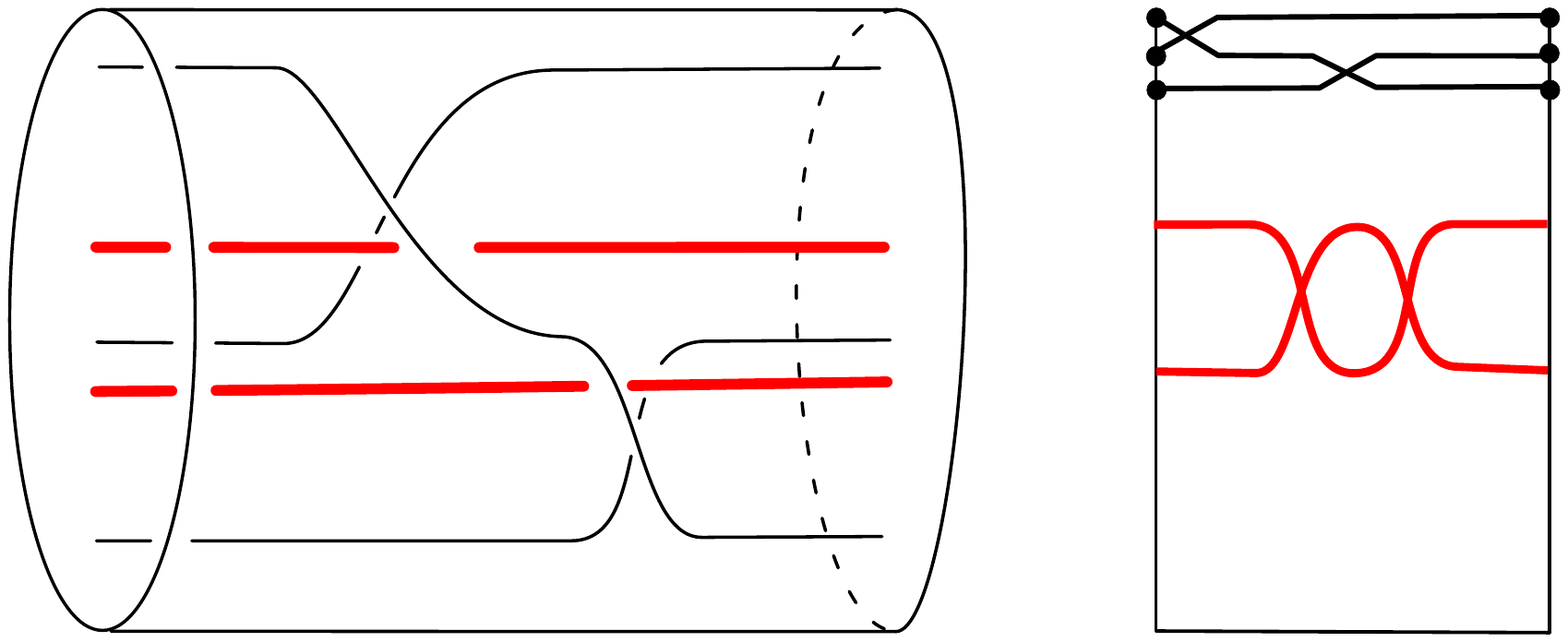}}
            \put(8,4.5){$\longrightarrow$}

            \put(6,6.2){$b^{\prime}$}
            \put(5.5,5.5){$b_{1}$}
        \end{picture}
        \caption{A smooth map $h_{1}$ for the trivial braid $b_{1}$ when $n=2$. In this case, $b^{\prime}$ is given as $\sigma_{1} \sigma_{2}$.}
        \label{mu}
    \end{figure}

From the construction of the stable map $S^{3} \to \mathbb{R}^{2}$ in \cite{K2026}, we obtain the following.

\begin{lemma}

    \label{sub3}

    Let $W$ be a solid cylinder $\mathbb{D}^2 \times [0,1]$ and $b_0 \in B_{n+1}$ any $(n+1)$-braid.
    Then there exists a smooth map $h_0$ from $W$ to $\mathbb{R}^2$ satisfying the following
    conditions.
    \begin{enumerate}
      \item[(i)] $h_0$ has no cusp points and no singular fibers of type $\mathrm{II}^3$
      \item[(ii)] $h_0|_{\mathbb{D}^2 \times \{\varepsilon\}} = \psi$ for $\varepsilon = 0, 1$,
      where $\psi$ is the Morse function with $n$ saddle points shown in Figure~\ref{tegaki2}
      \item[(iii)] $S_0(h_0)$ is an $(n+1)$-braid representing $b_0 \in B_{n+1}$, and $S_1(h_0)$ is the trivial $n$-braid, both relative to $\mathbb{D}^2 \times \{0,1\}$
      \item[(iv)] $S_0(h_0) \cap (\mathbb{D}^2 \times \{\varepsilon\})$ and $S_1(h_0) \cap (\mathbb{D}^2 \times \{\varepsilon\})$ consist of the local maxima and the saddle points of $\psi$, respectively.
    \end{enumerate}
\end{lemma}

\begin{proof}

    Following the construction of the map $N_1 \to R_1$ in the proof of \cite[Theorem~$1.1$]{K2026}, we obtain a smooth map $h_0$ whose set of definite fold points represents $b_0$. In that construction, the functions at both ends coincide with $\psi$ in Figure~\ref{tegaki2}, which gives (ii). Moreover, the saddle points are moved only by changing their heights. Each half-twist is applied to a connected component of the part where $\psi$ is greater than a given value, and if this component contains a saddle point, the saddle point is placed at the center of the rotation. Hence the saddle points do not braid with each other, and $S_1(h_0)$ is the trivial braid.
\end{proof}

\begin{proof}[Proof of Theorem~\ref{main2}]
In this proof, we follow the notation used in the proof of Theorem~\ref{main1}. Let $L_{0}$ and $L_{1}$ be links in $S^{3}$. We isotope $L_{0}, L_{1}$ so that $L_{0}, L_{1}$ is given as the closures of braids $b_{0}, b_{1}$, respectively. By applying stabilization to $b_{0}$ or $b_{1}$ if necessary, we may assume that $b_{0}$ is an $(n+1)$-braid and $b_{1}$ is an $n$-braid.


By Lemma~\ref{sub2}, we obtain a smooth map $h_1$ from $\mathbb{D}^2 \times [0,1]$ to $\mathbb{R}^2$ such that $S_1(h_1)$ represents $b_1 \in B_{n}$. Let $b' \in B_{n+1}$ be the braid represented by $S_0(h_1)$. By Lemma~\ref{sub3}, we obtain a smooth map $h_0$ from $\mathbb{D}^2 \times [0,1]$ to $\mathbb{R}^2$ such that $S_0(h_0)$ represents $(b')^{-1} b_0 \in B_{n+1}$ and $S_1(h_0)$ is the trivial braid in $B_{n}$. Since $h_1|_{\mathbb{D}^2 \times \{1\}} = \psi = h_0|_{\mathbb{D}^2 \times \{0\}}$, the maps $h_1$ and $h_0$ are glued into a smooth map $h$ from $\mathbb{D}^2 \times [0,2]$ to $\mathbb{R}^2$, where $h_0$ is regarded as a map on $\mathbb{D}^2 \times [1,2]$. Then $S_0(h)$ represents $b' \cdot (b')^{-1} b_0 = b_0 \in B_{n+1}$, and $S_1(h)$ represents $b_1 \cdot e = b_1 \in B_{n}$, where $e$ is the identity element in $B_{n}$.

We connect this map with a smooth map from a solid cylinder $N_{2}$ into the region $R_{2}$ naturally induced from $h_{1} |_{\mathbb{D}^{2} \times \{ 0 \}}$ and $h_{0} |_{\mathbb{D}^{2} \times \{ 1 \}}$ by the product structure of $N_{2}$. As a result, we obtain a smooth map from a solid torus into an annulus. Further, connecting this smooth map on a solid torus and the natural projection $g_{2} : V_{2} \to E$, we finally obtain a smooth map $f : S^{3} \to A \cup E \subset \mathbb{R}^{2}$. By construction, this map $f$ has no cusp points, and its singular set consists of definite and indefinite fold points. Also, $f$ satisfies the global conditions (5) and (6). Therefore, by \cite{Levin'65}, $f$ is a stable map. Moreover, the set of definite fold points forms a link isotopic to a given $L_{0}$, the set of indefinite fold points forms a link isotopic to a given $L_{1}$, and $f$ has no singular fibers of type $\mathrm{I\hspace{-1.2pt}I^{3}}$.

\end{proof}

\begin{example}

Let $L_{0}$ be the figure-eight knot and $L_{1}$ the trefoil knot in $S^{3}$. A stable map $f \colon S^{3} \to \mathbb{R}^{2}$ such that $S_{0}(f)$ is isotopic to $L_{0}$ and $S_{1}(f)$ is isotopic to $L_{1}$ is illustrated in Figure~\ref{tre8}. Here, $b_{i}$ are $3$-braids and $b'_{i}$ are $2$-braids for $i=1,2,3$, defined by $b_{1} = (\sigma_{1}\sigma_{2}\sigma_{1})^{3}$, $b_{2} = (\sigma_{1}^{-1}\sigma_{2}^{-1}\sigma_{1}^{-1})^{3}$, $b_{3} = (\sigma_{1}\sigma_{2}^{-1})^{2}$, and $b'_{1} = \sigma_{1}^{3}$, with $b'_{2}, b'_{3}$ being trivial $2$-braids.
In particular, their closures give $\widehat{b_{3}} = L_{0}$ and $\widehat{b'_{1}} = L_{1}$.

    \begin{figure}[htbp]
        \setlength\unitlength{1truecm}
        \begin{picture}(15,20)(0,0)
            \put(-4,0){\includegraphics[width=1\textwidth,clip]{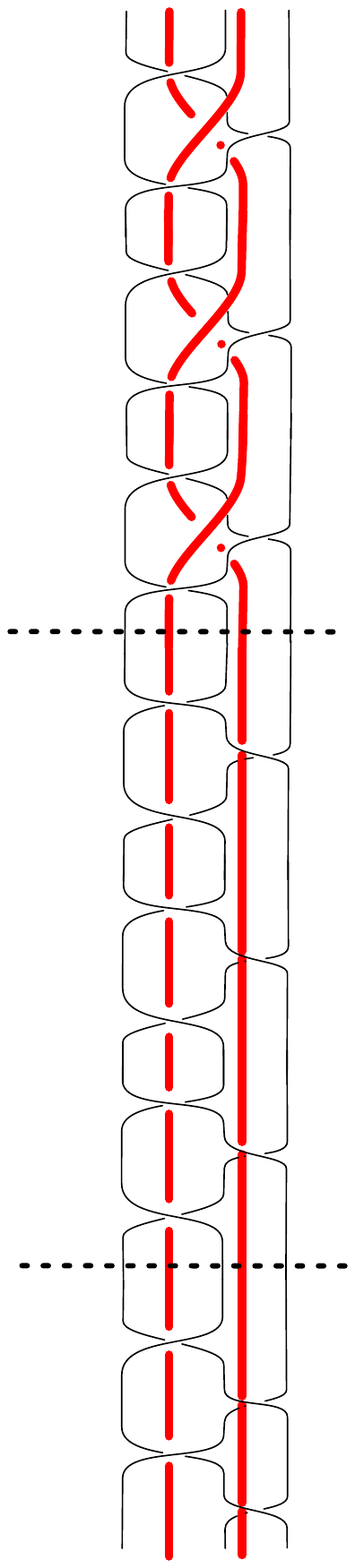}}
                    \put(0,14){$b_{1}$, $b'_{1}$}
                    \put(0,7){$b_{2}$, $b'_{2}$}
                    \put(0,1){$b_{3}$, $b'_{3}$}
                \put(4.5,9){$\longrightarrow$}
            \put(5,2){\includegraphics[width=0.8\textwidth,clip]{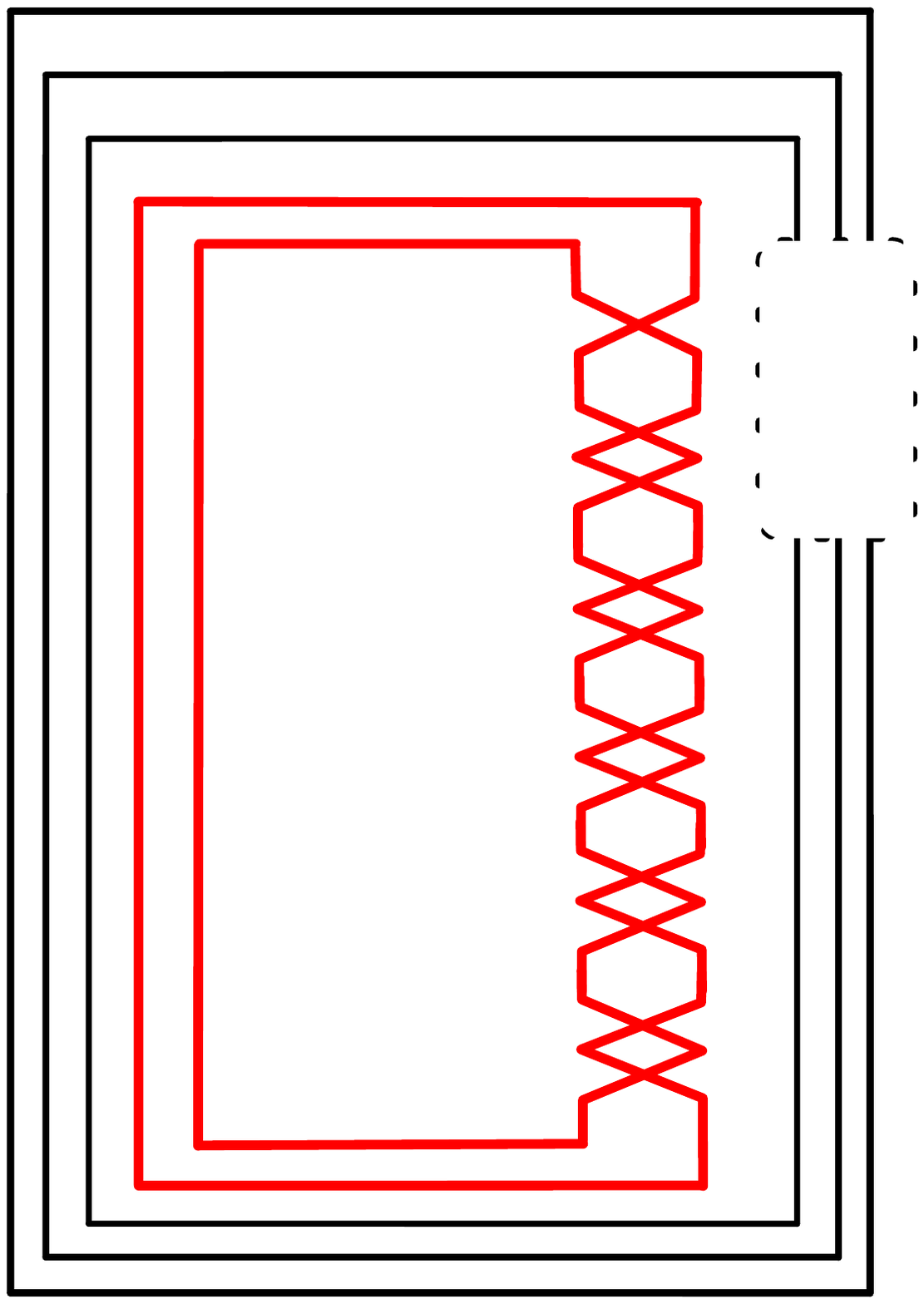}}
                \put(14,11.5){$B$}
        \end{picture}
        \caption{A stable map $f : S^{3} \to \mathbb{R}^{2}$ such that $S_{0} (f)$ is isotopic to the figure-eight knot and $S_{1} (f)$ is isotopic to the trefoil knot.}
        \label{tre8}
    \end{figure}

    \begin{figure}[htbp]
        \setlength\unitlength{1truecm}
        \begin{picture}(15,10)(0,0)
            \put(3,0){\includegraphics[width=0.5\textwidth,clip]{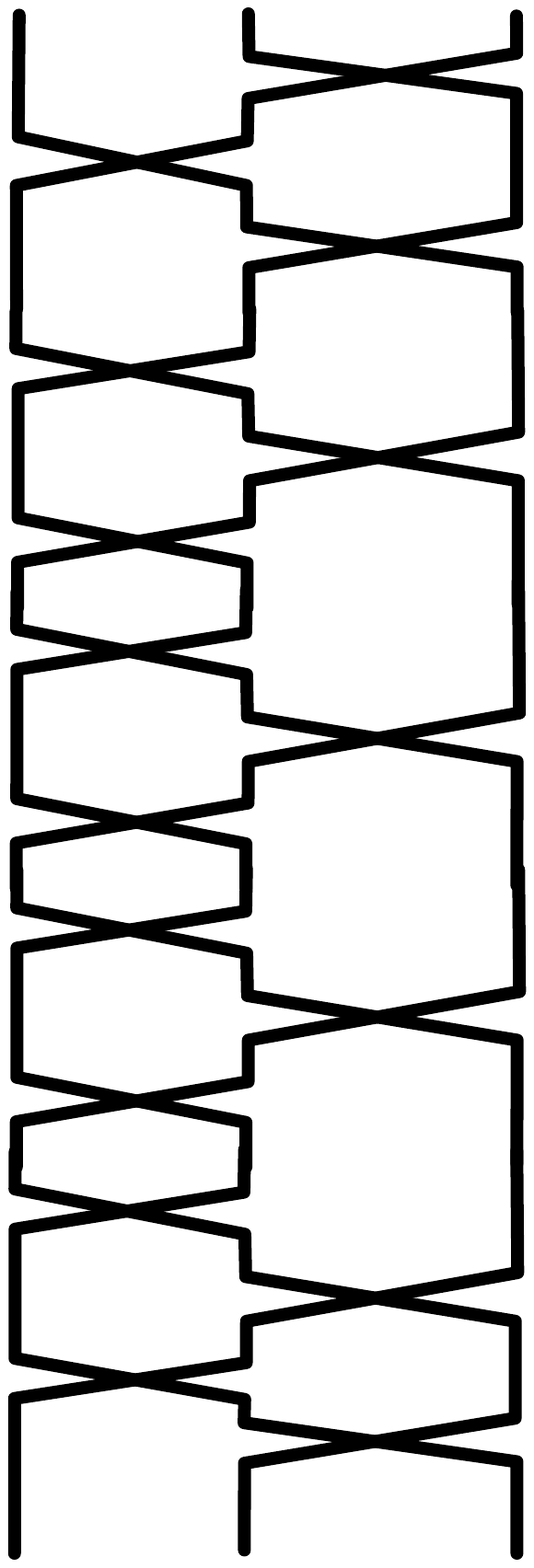}}
        \end{picture}
        \caption{The crossings in the box $B$.}
        \label{tre8tB}
    \end{figure}
\end{example}






\section*{Acknowledgements}
The author would like to express his sincere gratitude to his supervisor, Professor Kazuhiro Ichihara, for his invaluable guidance, continuous encouragement, and insightful suggestions throughout this work.



\providecommand{\bysame}{\leavevmode\hbox to3em{\hrulefill}\thinspace}
\providecommand{\MR}{\relax\ifhmode\unskip\space\fi MR }
\providecommand{\MRhref}[2]{%
  \href{http://www.ams.org/mathscinet-getitem?mr=#1}{#2}
}
\providecommand{\href}[2]{#2}

\end{document}